\documentclass[a4paper,12pt,draft]{article}

\usepackage{amsmath}
\usepackage{amssymb}
\usepackage{amsthm}
\usepackage{mathtools}
\usepackage{enumitem}
\usepackage{empheq}
\usepackage{authblk}
\usepackage[final]{hyperref}
\hypersetup{
	colorlinks=true,
	citecolor=black,
	filecolor=black,
	linkcolor=blue,
	urlcolor=blue
}

\newtheorem{theorem}{Theorem}[section]
\newtheorem{lemma}[theorem]{Lemma}
\newtheorem{proposition}[theorem]{Proposition}

\newtheorem{corollary}[theorem]{Corollary}

\theoremstyle{remark}
\newtheorem{remark}[theorem]{Remark}

\def\R{\mathbb{R}}
\def\d{\partial}
\def\dif{{\mathrm d}}
\def\ut{\tilde{u}}
\def\rt{\tilde{\rho}}
\def\vt{\tilde{v}}
\def\ep{\varepsilon}

\def\hv{\hat{v}}
\def\hu{\hat{u}}

\def\mB{\mathcal{B}}
\def\mD{\mathcal{D}}
\def\mE{\mathcal{E}}
\def\mF{\mathcal{F}}

\def\mL{\mathcal{L}}
\def\mQ{\mathcal{Q}}
\def\mR{\mathcal{R}}

\def\lpp#1{\|#1\|}

\newcommand{\qu}{\quad}

\begin{document}

\title{
On the stability of the spherically symmetric solution
to an inflow problem for an isentropic model of compressible viscous fluid}
\date{}
\author[$1$]{Yucong {\sc Huang}\footnote{Corresponding Author. Email Adress: huang@c.titech.ac.jp}}
\author[$2$]{Itsuko {\sc Hashimoto}} 
\author[$1$]{Shinya {\sc Nishibata}}

\affil[$1$]{\small Department of Mathematical and Computing Sciences, \protect\\ School of Computing, Institute of Science Tokyo, Tokyo 152-8552, Japan}
\affil[$2$]{\small Faculty of Mechanical Engineering, Kanazawa University, Ishikawa 920-1164, Japan}

\maketitle

\vspace{-6mm}

%------------
% Abstract
%------------
\begin{abstract}
    We investigate an inflow problem for the multi-dimensional isentropic compressible Navier-Stokes equations. The fluid under consideration occupies the exterior domain of the unit ball $\Omega=\{x\in\R^n\,\vert\, |x|\ge 1\}$ and a constant stream of mass is flowing into the domain from the boundary $\d\Omega=\{|x|=1\}$. It is shown in \cite{h-m21} that if the fluid velocity at the far-field is zero, then there exists a unique spherically symmetric stationary solution, denoted by $(\rt,\ut)$. In this paper, we show that either $\rt$ is monotone increasing or $\rt$ attains a unique global minimum and this is classified by the boundary condition of density. In addition, we derive a set of spatial decay rates for $(\rt,\ut)$ which allows us to prove a time asymptotic stability of $(\rt,\ut)$ using the energy method. More specifically, we prove this under small initial perturbation on $(\rt,\ut)$ provided that the density at the far-field is strictly positive but suitably small, in other words, the far-field state of the fluid is not vacuum but suitably rarefied. The main difficulty for the proof is the boundary terms that appears in the a-priori estimates due to the in-flowing boundary condition. We resolve this issue by reformulating the problem in Lagrangian coordinate system.
\end{abstract}

\vspace{-6mm}

{\hypersetup{linkcolor=blue}
	\tableofcontents
}

\section{Introduction}
 The
Navier-Stokes equation
for the isentropic motion
of  compressible viscous gas
in the Eulerian coordinate
 is  the system of equations given by
\begin{subequations}
\label{ns}
\begin{gather}
\d_t \rho  + \textrm{div} (\rho U) = 0, \label{ns1} \\
\rho \{ \d_t U + (U \cdot \nabla) U \} =
\nu \Delta U +(\nu + \lambda) \nabla \textrm{div} U
- \nabla P(\rho).
\label{ns2}
\end{gather}
\end{subequations}
We study the asymptotic behaviour of a solution
$(\rho, U)$ to (\ref{ns}) in an unbounded exterior domain %
$\Omega \vcentcolon= \{ z \in \R^n \; \vert \; |z| > 1 \}$, where $n\ge 2$ is the spatial dimension. Here $\rho>0$ is the mass density; $U=(U_1,\dots,U_n)$ is the velocity of gas; $P(\rho) = K\rho^\gamma \; (K>0, \gamma \ge 1)$ is the pressure
with the adiabatic exponent $\gamma$; $\nu$ and $\lambda$ are constants called viscosity-coefficients satisfying $\nu>0$ and $2\nu + n \lambda \ge 0$. In the equations (\ref{ns}), we use notations
\begin{gather*}
    \textrm{div} U \vcentcolon= \sum_{i=1}^n \d_i U_i, \qquad \nabla P \vcentcolon= ( \d_1 P, \dots, \d_n P ), \qquad \d_i \vcentcolon= \d_{z_i},\\
    (U \cdot \nabla) U \vcentcolon= \big((U \cdot \nabla) U_1, \dots, (U \cdot \nabla) U_n\big), \qquad (U \cdot \nabla) U_j \vcentcolon= \sum_{i=1}^n U_i \d_i U_j,\\
    \Delta U \vcentcolon= (\Delta U_1, \dots, \Delta U_n), \qquad \Delta U_j \vcentcolon= \sum_{i=1}^n \d_i^2 U_j,
\end{gather*}
It is assumed that the initial data is spherically symmetric. Namely, for $r \vcentcolon= |z|$
\begin{equation*}
\rho_0(z) = \rho_0(r), \quad U_0(z) = \dfrac{z}{r} u_0(r),
\end{equation*}
where $u_0(r):[1,\infty)\to \R$ is the radial component of $U_0(z)$. Under these assumptions, it is shown in
\cite{itaya85} that
the solution $(\rho, U)$ is also spherically symmetric, where the spherically symmetric solution is a solution to (\ref{ns}) in the form of
\begin{equation}
\rho(z,t) = \rho(r,t), \quad U(z,t) = \frac{z}{r} u(r,t),
\label{sym}
\end{equation}
where $u(r,t):[1,\infty)\times [0,\infty)\to \R$ is the radial component of $U(z,t)$.
Substituting (\ref{sym}) in (\ref{ns}), we reduce the system (\ref{ns}) to that of the equations for $(\rho, u)(r,t)$, which is given by 
\begin{subequations}
\label{nse}
\begin{gather}
\rho_t + \frac{(r^{n-1} \rho u)_r}{r^{n-1}} = 0, \label{nse1} \\
\rho (u_t + uu_r) = \mu \Big( \frac{(r^{n-1} u)_r}{r^{n-1}}
\Big)_r
- P(\rho)_r, \label{nse2}
\end{gather}
\end{subequations}
where $\mu := 2\nu + \lambda $ is a positive constant.
 The initial data
to (\ref{nse})
is prescribed to be spatial
asymptotically constants,
\begin{subequations}
\label{iac}
\begin{gather}
\rho(r,0) = \rho_0(r)>0, \quad u(r,0) = u_0(r), \label{ic}\\
\lim_{r \to \infty} (\rho_0(r), u_0(r)) = (\rho_+, u_+), \quad 0<\rho_+ <\infty, \ \ u_+\in\R. \label{asymp}
\end{gather}
\end{subequations}
In the present paper, we consider the model where fluid is flowing into the domain at a constant flux from an inner sphere centred at the origin of radius $r=1$. Its corresponding boundary condition is given by
\begin{equation}\label{ub}
 u(1,t) = u_b, \quad \rho(1,t)= \rho_b \qquad  \text{where $u_b>0$ and $\rho_b>0$.} 
\end{equation}
At the far-field, a constant non-vacuum state for the fluid is imposed. Namely,
\begin{equation}\label{ffc}
\lim\limits_{r\to \infty}(\rho,u)(r,t) \to (\rho_+,u_+) \qquad \text{for all }  t\ge 0,
\end{equation}
where $\rho_+>0$ and $u_+$ are the constants in \eqref{asymp}. Moreover, it is assumed that the initial data (\ref{ic}) is compatible with the boundary data (\ref{ub}), namely
\begin{subequations}\label{compa}
\begin{gather} 
\rho_0(1)=\rho_b, \qquad \big\{ u_b (\rho_0)_r + \rho_b (u_0)_r + (n-1) \rho_b u_b \big\}\big\vert_{r=1}=0,  \label{compa1} \\ 
u_0(1) = u_b, \qquad \Big\{ \rho_0 u_0 (u_0)_r + \mu \Big( \frac{(r^{n-1}u_0)_r}{r^{n-1}}\Big)_r - P(\rho_0)_r \Big\} \bigg|_{r=1} = 0. \label{compa2}
\end{gather}
\end{subequations}
We remark that since the characteristic speed of \eqref{nse1} is positive on the boundary due to $u_b>0$, a boundary condition for the density $\rho$ around $r=1$ is necessary for the well-posedness of the initial boundary value problem \eqref{nse}, \eqref{iac} and \eqref{ub}.

This initial boundary value problem is
formulated to study the behaviour of compressible
viscous gas flowing from the inner sphere.
We study the case where the time asymptotic state of the solution
to the problem (\ref{nse})--(\ref{ub}) is the stationary solution, which is a solution to (\ref{nse}) independent of time $t$, satisfying the same conditions (\ref{asymp}) and (\ref{ub}). Hence the stationary solution $(\rt(r), \ut(r))$ satisfies the system of equations
\begin{subequations}
\label{st}
\begin{gather}
\frac{1}{r^{n-1}}(r^{n-1} \rt \ut)_r = 0, \label{st1} \\
\rt \ut \ut_r = \mu \Big( \frac{(r^{n-1} \ut)_r}{r^{n-1}} \Big)_r
- P(\rt)_r \label{st2}
\end{gather}
\end{subequations}
and the boundary and spatial asymptotic conditions
\begin{equation}\label{stbdry}
 (\rt,\ut)(1) = (\rho_b,u_b), \quad
 \lim_{r \to \infty} (\rt(r), \ut(r)) = (\rho_+, u_+).
\end{equation}
Multiplying $r^{n-1}$ on (\ref{st1}) and integrating the resultant equality over $(1,r)$, we obtain
\begin{equation}\label{stu}
\ut(r) = \rho_b u_b \dfrac{r^{1-n}}{\rt(r)}, \qquad \text{for $r\ge1$.}
\end{equation}
Since $n\ge2$ and $\rho_+>0$, if a solution $(\rt,\ut)(r)$ exists, then it is necessary that $\ut(r)\to 0$ as $r\to\infty$. Thus we impose further that 
\begin{equation}\label{u+}
	u_+ = 0
\end{equation}
as a necessary condition for the
existence of the stationary solution. This means the far-field velocity of the fluid is at rest.
\section{Main results}\label{sec:MR}
First, we state the existence, uniqueness and property of the stationary solution for the problem \eqref{st}--\eqref{u+}, which are summarised in the following lemma.
\begin{lemma} \label{le:st}
There exists $\ep=\ep(\rho_+,\mu,\gamma,K,n)>0$ and $C=C(\mu,\gamma,K,n)>0$ such that if $|\rho_b-\rho_+|+ u_b\le \ep$, then a unique solution $(\tilde{\rho}, \tilde{u})$ to the problem \eqref{st}--\eqref{u+} exists in a certain neighborhood of $(\rho_+, 0)$. Moreover, for $r\ge 1$
\begin{gather} 
%\label{uhugou} \ut(r) = \rho_b u_b \frac{r^{1-n}}{\rt(r)}, \quad \ut(r)>0, \quad \ut_r(r)<0, \quad \ut_{rr}(r)<0,\\
\label{rho}|\ut(r)|\le C u_b r^{1-n}, \qquad |\rt(r)-\rho_+|\leq C\big\{ |\rho_b-\rho_{+}| + \rho_{+}^{2-\gamma} u_b^2 \big\}   r^{-2n+2},\\
\label{ur} \left|\frac{d \ut}{dr}\right| %\leq C \Big\{ u_b + \rho_{+}^{\gamma-1}|\rho_b-\rho_+| + \rho_{+} u_b^2 + u_b^2 |\rho_b-\rho_+| + \rho_{+}^2 u_b^3 \Big\}r^{-n}
\leq C \Big\{ u_b + \rho_{+}^{\gamma-1}|\rho_b-\rho_+| \Big\}r^{-n}, \quad  
\left|\frac{d\rt}{dr}\right|%\leq C\Big\{ \rho_{+}^2 u_b + \rho_{+}^{\gamma} \frac{|\rho_{b}-\rho_{+}|}{u_b} + \rho_{+} u_b |\rho_b-\rho_{+}| + \rho_{+}^3 u_b^2 \Big\}   r^{-2n+1}
\leq C\Big\{ \rho_{+}^2 u_b + \rho_{+}^{\gamma} \frac{|\rho_{b}-\rho_{+}|}{u_b} \Big\}   r^{-2n+1}. 
\end{gather}
Furthermore, the sign of $\rho_b-\rho_+$ determines the behaviour of density $\rt(r)$ as follows
\begin{enumerate}[label=\textnormal{(\Roman*)},ref=\textnormal{(\Roman*)}]
    \item\label{item:I} if $\rho_b \ge \rho_+$, then there exists a unique point $r_{\ast}>1$ such that $\rt(r)$ is strictly decreasing in $1\le r< r_{\ast}$ and strictly increasing in $r> r_{\ast}$ with
    \begin{equation*}
        \min\limits_{r\ge 1} \rt(r) = \rt(r_{\ast}), \qquad  \lim\limits_{r\to\infty} \rt(r) = \rho_+.
    \end{equation*}
    \item\label{item:II} if $\rho_b < \rho_{+}$, then $\rho(r)$ is strictly increasing in $r\ge 1$.
\end{enumerate}
\end{lemma}
The existence of the stationary solution satisfying \eqref{u+} is first proved in \cite{h-m21}.
In Section \ref{sec:st}, we review the well-posedness theorem stated in \cite{h-m21}. Then, based on their result, we derive the improved decay rates \eqref{rho}--\eqref{ur}. Moreover, we prove that if $(\rt,\ut)$ solves \eqref{st}--\eqref{u+}, then the properties \ref{item:I}--\ref{item:II} in Lemma \ref{le:st} must hold. The two major improvements we obtained for the decay estimates \eqref{rho}--\eqref{ur} are as follows: 1. a stronger decay rate on the order of $r$ is proved to hold in contrast to the one stated in \cite{h-m21}; 2. The dependence on $\rho_{+}$ is explicitly determined. These two refinements play an essential role in the analysis of the time asymptotic stability in Section \ref{sec:prioriE}. This is one of the main results in the present paper and it is stated in the following theorem. 
\begin{theorem}
\label{mth}
Let $\sigma$ be an arbitrary constant
satisfying $0<\sigma<1$.
Suppose $\gamma>1$ and the initial data satisfying \eqref{compa} belongs to the function space,
\begin{subequations}\label{IDReg}
\begin{align}
 &r^{\frac{n-1}{2}}(\rho_0-\rt), \quad r^{\frac{n-1}{2}} (\rho_0-\rt)_r \ \in L^2(1,\infty), \qquad & \rho_0 \in \mathcal{B}^{1+\sigma}\big([1,\infty)\big), \\ 
 & r^{\frac{n-1}{2}} (u_0-\ut), \quad r^{\frac{n-1}{2}} (u_0-\ut)_r \ \in L^2(1,\infty), \qquad & u_0\in \mathcal{B}^{2+\sigma}\big([1,\infty)\big).
\end{align}
\end{subequations}
Then there exist constants $ \alpha= \alpha(\mu,\gamma,K,n)$ which is independent of $\rho_+$, $\rho_b$, $u_b$ and $\ep_0=\ep_0(\rho_+,\mu,\gamma,K,n)>0$ such that if
\begin{equation*}
    \rho_{+}\in (0,\alpha], \qquad  |\rho_b-\rho_+|\le u_b^2, \qquad u_b +\|r^{\frac{n-1}{2}}(\rho_0-\rt,u_0-\ut)\|_{H^1} \le \ep_0,
\end{equation*}
then for an arbitrary $T > 0$, there is a unique classical solution $(\rho, u)$ to the equations \eqref{nse} with initial data \eqref{compa} and boundary data \eqref{stbdry}, \eqref{u+}. Moreover,
\begin{gather*}
\nonumber
(\rho, u)\in \mathcal{B}^{1+\frac{\sigma}{2},1+\sigma}\times \mathcal{B}^{1+\frac{\sigma}{2},2+\sigma}\big([0,T]\times [1,\infty)\big),\\
(\rho-\rt, u-\ut)\in C\big([0, \infty); H^1(1,\infty)\big)
\end{gather*}
and $(\rho, u)$ converges to the stationary solution $(\rt,\ut)$ as time tends to infinity. Precisely,
\begin{align}\label{Tasymp}
\lim_{t\to +\infty}\sup_{r \ge 1}|(\rho, u)(t,r)-(\rt,\ut)(r)|=0.
\end{align}
\end{theorem}
\begin{remark}
In Theorem \ref{mth}, only the isentropic case $\gamma>1$ is considered. The stability of the stationary solution $(\rt,\ut)$ for isothermal case $\gamma=1$ is an open problem that we plan to investigate further in the future. Moreover, we have also imposed in Theorem \ref{mth} that the far-field reference density $\rho_{+}$ can not be too large, in the sense that it is bounded above by some given constant. For several cases of the inflow problem posed in $1$-dimensional half line, which is studied in \cite{m-n01}, this condition is not imposed. Hence we see that this difficulty is unique to the multi-dimensional spherically symmetric model. The question of whether this restriction can be relaxed is the subject of ensuing research and we believe that the solution to this problem could provide further insight towards the stability of spherically symmetric in-flowing fluid.    
\end{remark}

\paragraph{Related results.} The compressible Navier-Stokes equation is an important subject in mathematical physics. Particularly, the problems of out-flowing or in-flowing boundary have gained traction in recent decades. We state several previous results, which are relevant to the present paper.

First, for a comprehensive survey of the mathematical theory of compressible Navier-Stokes equation, we refer readers to the book \cite{kaz90} by S.~N. Antontsev, A.~V. Kazhikhov and V.~N. Monakhov. The first breakthrough research on the large time stability of solution is done by A.~Matsumura and T.~Nishida in \cite{m-n83}, where they consider the heat-conductive compressible flow for a general $3$-dimensional exterior domain with adhesion boundary condition ($u\vert_{\d\Omega}=0$ and $\d_{n}\theta\vert_{\d\Omega}=0$ where $\theta$ is the 
absolute 
temperature). In this case the stationary solution is $(\tilde{\rho},0,\bar{\theta})$ where $\tilde{\rho}=\tilde{\rho}(x)$ is a positive function of spatial variable and $\bar{\theta}$ is a positive constant. Under smallness assumptions on the initial perturbation and external forces, it is proved that $(\tilde{\rho},0,\bar{\theta})$ is stable globally in time.

When the equation is spherically symmetric with adhesion boundary condition, a pioneering work has been done by N.~Itaya \cite{itaya85}, which establishes the global-in-time existence of a unique classical solution on a bounded annulus domain. In this work, no smallness assumption on the initial data is imposed. Later, T.~Nagasawa studies the asymptotic state for the same
problem in \cite{nagasawa}. The paper \cite{itaya85} has motivated a sequence of developments on the topic of spherically symmetric solution. For instance, A.~Matsumura in \cite{matsu92} constructed a spherically symmetric classical solution to the isothermal model with external forces on a bounded annular domain. Moreover, he also shows the convergence rate to the stationary solution as time tends to infinity is exponential. Subsequently, this result has been extended to the isentropic and heat-conductive models by K.~Higuchi in \cite{higuchi92}. The well-posedness of spherically symmetric solution in an unbounded exterior domain is first obtained by S.~Jiang \cite{jiang96}, where the global-in-time existence of a uniqueness classical solution is shown. In addition, a partial result on the time asymptotic stability is proved in \cite{jiang96} where, for $n=3$, $\|u(\cdot,t)\|_{L^{2j}}\to 0$ as $t\to\infty$ with any fixed integer $j\ge 2$. Later, this restriction on the long time stability was fully resolved by T.~Nakamura and S.~Nishibata in \cite{n-n08}, where a complete stability theorem was obtained for both $n=2$ and $3$ with large initial data. We also refer to the paper \cite{NNY} by T.~Nakamura, S.~Nishibata and S.~Yanagi, where the time asymptotic stability of the spherically symmetric solution for the isentropic flow is established with large initial data and external forces.

For the general outflow or inflow problem, the stationary solution becomes non-trivial. This leads to a variety of physically interesting time asymptotic behaviours for the solutions. A.~Matsumura in \cite{m2001} starts the first investigation of this problem for the isentropic model posed on the one dimensional half-space domain. Several kinds of boundary conditions were considered in \cite{m2001}, which includes inflow, outflow and adhesion boundary conditions. He formulated
conjectures on the classification of asymptotic behaviours of the solutions in different cases subject to the relation between the boundary data and the spatial asymptotic data. Then the stability theorems for some cases of inflow problem were established by A.~Matsumura and K. Nishihara in \cite{m-n01} where they employed the Lagrangian mass coordinate. Following this work, S.~Kawashima, S.~Nishibata and P.~Zhu in \cite{k-n-z03} further examine the outflow problem in one dimensional half space. They prove the long time stability of solutions with small initial perturbation with respect to the stationary solution. A detailed desicription of the convergence rate towards the stationary
solution is found in the work of T.~Nakamura, S.~Nishibata and T.~Yuge \cite{NNY}.

For the non-isentropic inflow problem in the half line, T.~Nakamura and S.~Nishibata in \cite{NN} demonstrate the time asymptotic stability of stationary solutions under a small initial perturbation, for both the subsonic and transonic cases. 

The research on the outflow and inflow problems for the spherically symmetric solution in an unbounded exterior domain has been relatively new subject. Under the assumption that velocity at boundary, $|u_b|$ is sufficiently small, I.~Hashimoto and A.~Matsumura in \cite{h-m21} employ the iteration method to obtain the existence of a unique spherically symmetric stationary solutions for both inflow and outflow problems in an exterior domain. More recently, I.~Hashimoto, S.~Nishibata, S.~Sugizaki in \cite{h-n-s23} show the stability of spherically symmetric stationary solution for outflow problem under small initial perturbation. Subsequently, Y. Huang and S. Nishibata in \cite{h-n24} considers the large initial perturbation for the outflow problem and they prove that the stationary solution remains stable in large time without any smallness assumption on the initial data.  

\paragraph{Outline of the paper.} In Section \ref{ssec:stExist}, we review the well-posedness of stationary solution $(\rt,\ut)$, which is originally proved in \cite{h-m21}. Based on their result, we derive in Section \ref{ssec:stRate}, the improved decay estimate Lemma \ref{lemma:etaEst}. This is needed for the proof of long time stability of $(\rt,\ut)$. Moreover, based on the boundary data $\rho_b$, we classify in Section \ref{ssec:extremum}, the extremum points of $\rt$. In Section \ref{sec:prioriE}, we derive the a-priori estimates of $(\rho-\rt,u-\ut)$, from which the long time stability of $(\rt,\ut)$ is implied. To do so, we reformulate the original equations \eqref{nse} into Lagrangian coordinates in Section \ref{ssec:reform}. Subsequently, using the energy method, we obtain the $L^{\infty}\big(0,T;H^1\big)$ a-priori estimates for $(\rho-\rt,u-\ut)$ in Sections \ref{ssec:REE}--\ref{ssec:H1psi}. The main reason for this coordinate transformation is that one encounters the boundary term $(u-\ut)_r^2\vert_{r=1}$, which poses difficulty in the Eulerian formulation. However this issue is resolved by utilising the distinctive structure of equations \eqref{nsl-hat} in Lagrangian formulation.

\paragraph{\bf Notations.} 
%Before stating the main result, 
We list the notations which are frequently used in the present paper.
\begin{enumerate}[label=\textnormal{(\arabic*)},ref=\textnormal{(\arabic*)}]
\item\label{item:cons} $C=C(\mu,\gamma,K,n)$ denotes a positive constant, which depends only on the parameters $\mu$, $\gamma$, $K$ and $n$. Moreover, we also use the notation $C(\rho_+)$ to denote positive constants that depends on $\rho_+$, in addition to $\mu$, $\gamma$, $K$ and $n$. We also define
\begin{align*}
    A\sim B \qquad \text{if and only if} \qquad C^{-1}B \le A \le C B.  
\end{align*}
\item For a non-negative integer $k \ge 0$, $H^k(\Omega)$ denotes the $k$-th order Sobolev space over $\Omega$ in the $L^2$ sense with the norm
\begin{equation*}
\| f \|_k\vcentcolon= \|f\|_{H^{k}(\Omega)} = \bigg( \sum_{|\alpha|=0}^{k} \int_{\Omega}\!\!  |\d_x^{\alpha} f (x)|^2 \, \dif x \bigg)^{\frac{1}{2}}.
\end{equation*}
We note also that $H^0 = L^2$ and denote $\lpp{\cdot}  \vcentcolon= \| \cdot \|_0$.
\item For a non-negative integer $k\ge 0$, $\mB^k(\Omega)$ denotes the space of all functions $f$ which, together with all their partial derivatives $\d_x^{i}f$ of orders $|i|\le k$, are continuous and bounded on $\Omega$. It is endowed with the norm
\begin{equation*}
\|u \|_{\mB^k(\Omega)}
\vcentcolon= \sum_{i=0}^k \sup_{x \in \Omega} |\d_x^i u(x)|.
\end{equation*}
Moreover, for $\alpha \in (0,1)$, $\mB^\alpha (\Omega)$ denotes the
space of bounded functions over $\Omega$ which have the
uniform H\"older continuity with exponent $\alpha$.
For an integer $k\ge 0$,  $\mB^{k+\alpha}(\Omega)$ denotes the space of the functions satisfying %
$\d_x^{i} u \in \mB^\alpha (\Omega)$ for all integer $i \in [0, k]$. It is endowed with the norm
$|\cdot|_{k+\alpha}$ is its norm defined by
\begin{equation*}
 \|u\|_{\mB^{k+\alpha}(\Omega)}
\vcentcolon= \sum_{i=0}^k \sup_{x \in \Omega} |\d_x^i u(x)|
+ \sup_{\begin{smallmatrix} x,x' \in \Omega \\ x \neq x' \end{smallmatrix}}
  \frac{|\d_x^k u(x) - \d_x^k u(x')|}{|x - x'|^\alpha}.
\end{equation*}
\item Denote $Q_T \vcentcolon= [0,\infty) \times [0, T]$. $\mB^{\alpha, \beta} (Q_T)$ denotes the space of the uniform H\"{o}lder continuous functions with the H\"{o}lder exponents $\alpha$ and $\beta$ with respect to $x$ and $t$, respectively. For integers $k$ and $l$, $\mB^{k+\alpha, l+\beta}(Q_T)$ denotes the space of the functions satisfying $\d_x^i u, \d_t^j u \in \mB^{\alpha, \beta}(Q_T)$ for all integers $i \in [0,k]$ and $j \in [0,l]$.
\end{enumerate}

\section{Stationary solution}\label{sec:st}
\subsection{Well-posedness of the stationary solution.}\label{ssec:stExist}
In this subsection, we discuss the existence and uniqueness of the stationary solution $(\rt,\ut)$  to (\ref{st}). Set $\vt\vcentcolon=1/\rt$ to be the specific volume, $v_{b}\vcentcolon= 1/\rho_b$ and $v_{+}= 1/\rho_{+}$. Owing to the far-field condition $v\to v_{+}$ as $r\to \infty$, we define the unknown $\eta \vcentcolon= \vt - v_+$. Then substituting \eqref{stu} into \eqref{st2} yields the ordinary differential equation
\begin{subequations}\label{eta0}
\begin{equation}
    m_b\mu \left(\frac{ \eta_r }{r^{n-1}}\right)_r=
p\big(v_+ + \eta \big)_r
+\frac{ m_b^2 v_+}{2}\Big(\frac{1}{r^{2(n-1)}}\Big)_r
+\frac{ m_b^2}{r^{n-1}}\left(\frac{\eta}{r^{n-1}}\right)_r,\label{ODE0}
\end{equation}
where $m_b\vcentcolon= \rho_b u_b$ and $p(v):=P\left(v^{-1}\right)$. Here $\eta$ satisfies the boundary and far-field conditions
\begin{equation}
\eta(1)=\eta_b\equiv v_b-v_+ \ \text{ and } \ \eta(r) \to 0 \quad \text{as} \quad r \to \infty. \label{far}
\end{equation}
\end{subequations}
Using \eqref{eta0}, we obtain a representation formula for $\eta$, which is suited for the fixed-point theorem for the contraction mapping. First, integrating \eqref{ODE0} by parts and using the far-field condition \eqref{far}, we have
%\begin{equation*}
%m_b \mu \frac{\eta_r(r)}{r^{n-1}} + p(v_{+}) - p(v_++\eta(r)) -\frac{m_b^2 v_+}{2r^{2(n-1)}} -\frac{m_b^2\eta(r)}{r^{2(n-1)}} +m_b^2(n-1)\int_r^{\infty}\!\!\frac{\eta(s)}{s^{2n-1}}d s =0.
%\end{equation*}
%Rewriting the above equation, we obtain
\begin{align}
\eta_r(r)= &
\frac{r^{n-1}}{m_b\mu}\big\{ p(v_++\eta(r))- p(v_+)\big\}
 +\frac{m_b v_+}{2\mu}\frac{1}{r^{n-1}}\nonumber\\
&+\frac{m_b\eta(r)}{\mu r^{n-1}}
-\frac{m_b(n-1)r^{n-1}}{\mu}\int_r^{\infty}\frac{\eta(s)}{s^{2n-1}}ds.\label{reform}
\end{align}
Let $\eta_b\vcentcolon= v_b-v_+$. Subtracting $\frac{p^{\prime}(v_+)}{\mu m_{b}} r^{n-1} \eta$ on both sides of \eqref{reform}, we have
\begin{equation}
\label{etai}
 \eta_{r}(r)-\displaystyle \frac{ p^{\prime}(v_+)}{\mu m_{b}} r^{n-1} \eta(r)=F[\eta](r), 
\end{equation}
where
\begin{align*}
&F[\eta](r):= \dfrac{m_b v_{+}}{2\mu} \dfrac{1}{r^{n-1}} + \dfrac{m_b}{\mu}\dfrac{\eta(r)}{r^{n-1}} - \dfrac{(n-1)m_b r^{n-1}}{\mu} \int_{r}^{\infty}\!\!\dfrac{\eta(s)}{s^{2n-1}}\,ds + \dfrac{r^{n-1}}{\mu m_b} N\big(\eta(r)\big),\\
& N(\eta(r)) :=p(v_++\eta(r))-p(v_+)-p'(v_+)\eta(r).
\end{align*}
Solving \eqref{etai}, we obtain the representation formula
\begin{gather}
\label{etae}
\eta(r)=\eta_b e^{\frac{\kappa}{m_b}(1-r^n)} + \int_{1}^{r} e^{\frac{\kappa}{m_b}(s^n-r^n)} F[\eta](s)\, {d} s\quad \text{where } \ \kappa:=\frac{-p'(v_+)}{n\mu}>0,\\
\label{etae2}
\eta_r(r) = - \frac{n\kappa }{m_b} r^{n-1} \eta(r) + F[\eta](r). 
\end{gather}
I. Hashimoto and A. Matsumura in \cite{h-m21} prove the existence of a unique classical solution to \eqref{eta0}, which is stated in the following lemma.
\begin{lemma}[I. Hashimoto and A. Matsumura, \cite{h-m21}] \label{lem:H-M21}
    Let $n\ge 2$. For an arbitrary $\rho_{+}>0$, there exist constants $\ep>0$ and $C(\rho_+)>0$ which depend on $\rho_{+}$, $\mu$, $K$, $\gamma$ and $n$, such that if $u_b+|\eta_b|<\ep$, then there exists a unique classical solution $\eta\in \mathcal{B}^2[1,\infty)$ satisfying
    \begin{equation}\label{h-m-rate}
        r^{n-1}|\eta(r)| \le C(\rho_{+})  \big\{ u_b^2 + |\eta_b| \big\} \qquad \text{ for } \ r\ge 1. 
    \end{equation}
\end{lemma}
In \cite{h-m21}, they consider the Banach space $X$, with its norm defined by
\begin{equation*}
    X \vcentcolon= \big\{ f\in C[1,\infty) \, \big\vert \,  \|f\|_{X} <\infty \big\} \qquad \text{where} \qquad \|f\|_{X} \vcentcolon= \sup\limits_{r\ge 1} r^{n-1}|\eta(r)|.
\end{equation*}
Then they define the map $\mathcal{T} \vcentcolon X \mapsto X$ by
\begin{align*}
    \mathcal{T} \vcentcolon \quad \eta \quad  \mapsto \quad \eta_b e^{\frac{-\kappa(r^n-1)}{m_b}} + \int_{1}^{r} e^{\frac{-\kappa(r^n-s^n)}{m_b}} F[\eta](s) {d} s.   
\end{align*}
It is shown in \cite{h-m21,h-n-s23} that there exist $\ep>0$ and $M>0$ such that if $u_b + |\eta_b| \le \ep$, then $\mathcal{T}$ is a contraction mapping in the subspace $S_{M}\vcentcolon= \{ f\in X \, \vert\, \|f\|_{X} \le M \}$. In particular, repeating the same argument presented in \cite{h-m21}, we obtain a small constant $\ep>0$ such that if $u_b+ |\eta_b|<\ep$, then the corresponding solution $\eta$ satisfies
\begin{equation}\label{E0}
    |r^{n-1}\eta(r)| \le \frac{v_+}{4} \quad \text{and} \quad \frac{3}{4} v_{+}  \le v(r) \le \frac{5}{4}v_{+} \qquad \text{for } \ r\ge 1.
\end{equation}

The decay rate given in \eqref{h-m-rate} is not sufficient to show the time asymptotic stability of $(\rt,\ut)$. Hence, one of the main aims of the present paper is to derive the improved decay rate for $(\rt,\ut)$, which is stated in Lemma \ref{le:st}. This will allow us to prove the a-priori estimate in Theorem \ref{thm:Lag}. As a result, the large time stability of $(\rt,\ut)$ is shown. 

%------------------------------------------------------------------
% START OF COMMENT: Dependence on $\rho_+$ of invariance estimate
%-------------------------------------------------------------------
\iffalse
\begin{remark}
    In \cite{h-m21}, the following iteration scheme is used
    \begin{equation*}
        \eta^{(0)}(r) = \eta_b e^{-\frac{\kappa}{m_b}(r^n-1)}, \qquad \eta^{(m+1)}(r) = \eta^{(0)}(r) + \int_{1}^r \!\! F\big[ \eta^{(m)} \big] (s) e^{-\frac{\kappa}{m_b}(r^n-s^n)}\, ds,
    \end{equation*}
    for $m\in \mathbb{N}$. If we keep track of the dependence on $\rho_+$, while repeat the same procedure employed in \cite{h-m21}, then there exists a constant $C=C(\mu,\gamma,K,n)>0$ such that
    \begin{equation*}
        \big\|\eta^{(m+1)}\big\|_{X}%=\sup\limits_{r\ge 1}\big| r^{n-1} \eta^{(m+1)}(r) \big| 
        \le |\eta_b| + C u_b^2 \rho_{+}^{-\gamma} + C u_b^2 \rho_{+}^{1-\gamma} \big\| \eta^{(m)} \big\|_{X} + C \rho_{+} \big\| \eta^{(m)} \big\|_{X}^2. 
    \end{equation*}
\end{remark}
\fi
%------------------------------------------------------------------
% END OF COMMENT: Dependence on $\rho_+$ of invariance estimate
%-------------------------------------------------------------------

\subsection{Decay rate for the stationary solution}\label{ssec:stRate}
In this subsection, we show the improved decay rates \eqref{rho}--\eqref{ur}. The proposition below will be useful and its proof is found in \cite{h-m21}. 

\begin{proposition}[I.~Hashimoto, A.~Matsumura \cite{h-m21}] \label{prop:H-M}
There exist positive constants $\ep=\ep(\rho_+,\mu,\gamma,K,n)$ and $C=C(\mu,\gamma,K,n)$ which are independent of $u_b$ and $\eta_b$, such that if $|\eta_b|+u_b \le \ep$ and $\sup_{r\ge 1}| r^{l} f(r)|< \infty$ for $r\ge 1$ and $\ell =1,\dotsc, 3n-3$, then
    \begin{equation*}
        r^{\ell} \bigg| \int_{1}^r e^{\frac{-\kappa(r^n-s^n)}{m_b}} f(s) ds \bigg| \le C \rho_{+}^{-\gamma} u_b \sup\limits_{r\ge 1} \big| r^{\ell} f(r) \big|.
    \end{equation*}
\end{proposition}
Using Lemma \ref{lem:H-M21} and Proposition \ref{prop:H-M}, we show that the stationary solution satisfies the following decay rate
\begin{lemma}\label{lemma:etaEst}
    Let $\eta$ be the solution obtained in Section \ref{ssec:stExist}. Then there exist positive constants $\ep=\ep(\rho_+,\mu,\gamma,K,n)$ and $C=C(\mu,\gamma,K,n)$ which are independent of $u_b$ and $\eta_b$, such that if $|\eta_b|+ u_b\le \ep$, then %$u_b\le C \rho_{+}^{\gamma}$.
    the stationary solution satisfies
    \begin{gather}
        \sup\limits_{r\ge 1}\big|r^{2(n-1)}\eta(r)\big|\le C\big\{ |\eta_b| + \rho_{+}^{-\gamma} u_b^2  \big\}, \qquad 
        \sup\limits_{r\ge 1}\big|r^{2n-1}\eta_r(r)\big| %\le C \Big\{ \rho_{+}^{\gamma} \frac{|\eta_b|}{u_b} + u_b  + \rho_{+}|\eta_b| u_b + \rho_{+} u_b^2 \Big\}
        \le C \Big\{ u_b + \rho_{+}^{\gamma} \frac{|\eta_b|}{u_b} \Big\}. \label{etaEst1}
        %\\ \sup\limits_{r\ge 1}\big|r^{2n}\eta_{rr}(r)\big| \le C \Big\{ 1 + \frac{|\eta_b|}{u_b^2} \Big\}.\label{etaEst2}
    \end{gather}
\end{lemma}
\begin{proof}
    By the definition of $F[\eta]$ and formula \eqref{etae}, we have
    \begin{align*}
        \eta(r) =& \underbrace{\eta_b e^{-\frac{\kappa}{m_b}(r^n-1)}}_{\textrm{(i)}} + \underbrace{\frac{m_b v_{+}}{2\mu} \int_{1}^r \!\! s^{1-n}  e^{\frac{\kappa}{m_b}(s^n-r^n)}\,ds}_{\textrm{(ii)}} + \underbrace{\frac{m_b}{\mu} \int_{1}^{r} \frac{\eta(s)}{s^{n-1}} e^{\frac{\kappa}{m_b}(s^n-r^n)}  \, ds}_{\textrm{(iii)}}\\
        & + \underbrace{\frac{(1-n)m_b }{\mu}\int_{1}^r \!\!\! \int_{s}^{\infty}\!\!\! \eta(\tau) \frac{s^{n-1}}{\tau^{2n-1}} e^{\frac{\kappa}{m_b}(s^n-r^n)} \, d\tau}_{\textrm{(iv)}} + \underbrace{\int_{1}^r\! \frac{s^{n-1}}{\mu m_b}  N\big( \eta(s) \big) e^{\frac{\kappa}{m_b}(s^n-r^n)} \, ds}_{\textrm{(v)}}.
    \end{align*}
    We estimate the right hand side of above equality term by term. First we impose that $m_b \le \kappa$. Then by \eqref{etae} and \eqref{E0}, this is equivalent to the condition
    \begin{equation}\label{urg}
        u_b \le C \rho_{+}^{\gamma},
    \end{equation}
    where $C=C(\mu,\gamma,K,n)>0$ is the constant independent of $\rho_+$, $\rho_b$ and $u_b$. Moreover, it holds that $-\frac{\kappa}{m_b} (r^n-1) \le -(r^n-1)$ for $r\ge 1$, since $\kappa>0$. Thus we have that $\exp\big(-\frac{\kappa}{m_b}(r^n-1)\big) \le \exp\big(-(r^n-1)\big)$. Using this, we get
    \begin{align*}
        |\textrm{(i)}| %=&  |\eta_b| \frac{r^{-2n}}{(n-1)^2}  (n-1)^2 r^{2n} e^{\frac{\kappa}{m_b}(1-r^n)} \\
        \le& |\eta_b| \frac{r^{-2n}}{(n-1)^2}  (n-1)^2 r^{2n} e^{ - (r^n-1)} %\le C |\eta_b| r^{-2n} 
        \le C |\eta_b| r^{-2(n-1)}.
    \end{align*}
    Integrating \textrm{(ii)} by parts and using Proposition \ref{prop:H-M}, we obtain
    \begin{align*}
        |\textrm{(ii)}| %=&\frac{m_b v_{+}}{2\mu} e^{-\frac{\kappa}{m_b}r^n} \bigg| \int_{1}^r s^{1-n} e^{\frac{\kappa}{m_b} s^n}\, ds  \bigg|\\ 
        %=& \frac{m_b^2 v_{+}}{2n\mu\kappa} e^{-\frac{\kappa}{m_b}r^n} \bigg| \int_{1}^r s^{2-2n} \frac{d}{ds}e^{\frac{\kappa}{m_b} s^n}\, ds  \bigg|\\
        %=& \frac{m_b^2 v_{+}}{2n\mu\kappa} e^{-\frac{\kappa}{m_b}r^n} \bigg| r^{2-2n} e^{\frac{\kappa}{m_b}r^n} - e^{\frac{\kappa}{m_b}} + 2(n-1)\int_{1}^r s^{1-2n} e^{\frac{\kappa}{m_b} s^n}\, ds  \bigg|\\
        =&  \frac{m_b^2 v_{+}}{2n\mu\kappa} \bigg| r^{2-2n} - e^{-\frac{\kappa}{m_b}(r^n-1)} + 2(n-1)\int_{1}^r s^{1-2n} e^{\frac{\kappa}{m_b} (s^n-r^n)}\, ds  \bigg| \\
        \le & C u_b^2 \rho_+^{-\gamma}   \{ r^{-2(n-1)} + r^{-2n+1} \} \le C u_b^2 \rho_+^{-\gamma} r^{-2(n-1)}.
    \end{align*}
    Using the upper bound \eqref{E0} and repeating the same argument used for estimating \textrm{(ii)}, we get
    \begin{align*}
        &|\textrm{(iii)}| \le \frac{m_b v_{+}}{4\mu} \bigg| \int_{1}^r s^{1-n} e^{\frac{\kappa}{m_b}(s^n-r^n)}\, ds \bigg|\\
        %=&\frac{m_b^2 v_{+}}{4n\mu\kappa} e^{-\frac{\kappa}{m_b}r^n} \bigg| r^{2-2n} e^{\frac{\kappa}{m_b}r^n} - e^{\frac{\kappa}{m_b}} + 2(n-1)\int_{1}^r s^{1-2n} e^{\frac{\kappa}{m_b} s^n}\, ds  \bigg| \\
        =& \frac{m_b^2 v_{+}}{4n\mu\kappa}  \bigg| r^{2-2n} - e^{-\frac{\kappa}{m_b}(r^n-1)} + 2(n-1)\int_{1}^r s^{1-2n} e^{\frac{\kappa}{m_b} (s^n-r^n)}\, ds  \bigg| \le C u_b^2 \rho_{+}^{-\gamma} r^{-2(n-1)}.
    \end{align*}
    Similarly, by Proposition \ref{prop:H-M} and \eqref{E0}, we also obtain the estimate
    \begin{align*}
        |\textrm{(iv)}| \le & \frac{(n-1)m_b}{\mu} \int_{1}^r s^{n-1} e^{\frac{\kappa}{m_b}(s^n-r^n)} \int_{s}^{\infty} \frac{|\eta(\tau)|}{\tau^{2n-1}} \, d\tau d s\\
        %\le & C m_b v_{+} \int_{1}^r s^{n-1} e^{\frac{\kappa}{m_b}(s^n-r^n)} \int_{s}^{\infty} \tau^{-3n+2} \, d\tau d s\\
        \le & C m_b v_{+} \int_{1}^r s^{-2n+2} e^{\frac{\kappa}{m_b}(s^n-r^n)} \,ds \le C u_b^2 \rho_{+}^{-\gamma}  r^{-2(n-1)}.
    \end{align*}
    By Taylor's theorem and upper and lower bounds \eqref{E0}, it holds that for an arbitrary $s\in [1,r]$ there exists $\xi_s\in (\frac{3}{4} v_{+}, \frac{5}{4}v_{+})$ such that 
    \begin{align}\label{NMVT}
        \big|N\big(\eta(s)\big)\big| = \frac{1}{2} p^{\prime\prime}(\xi_s) |\eta(s)|^2 \le C \rho_{+}^{\gamma+2} |\eta(s)|^2.
    \end{align}
    Using \eqref{NMVT}, Lemma \ref{lem:H-M21} and Proposition \ref{prop:H-M}, we get
    \begin{align*}
        |\textrm{(v)}| \le& C\frac{\rho_{+}^{\gamma+2}}{m_b}  \int_{1}^r s^{n-1} |\eta(s)|^2 e^{\frac{\kappa}{m_b}(s^n-r^n)}\, ds\\ 
        \le& C\frac{\rho_{+}^{\gamma+2}}{m_b} \sup\limits_{1\le \tau \le r} |\tau^{2(n-1)}\eta(\tau)| r^{-2(n-1)} r^{2(n-1)}\int_{1}^r s^{-n+1} |\eta(s)| e^{\frac{\kappa}{m_b}(s^n-r^n)}\, ds\\
        \le & C\frac{\rho_{+}^{\gamma+2}}{m_b} \sup\limits_{1\le \tau \le r} |\tau^{2(n-1)}\eta(\tau)| \cdot r^{-2(n-1)}  u_b \rho_{+}^{-\gamma}\sup\limits_{r\ge 1} |r^{n-1} \eta(r)|\\
        \le& \rho_{+} C(\rho_{+})  \big\{ |\eta_b| + u_b^2 \big\} r^{-2(n-1)} \sup\limits_{1\le \tau\le r} |\tau^{2(n-1)} \eta(\tau)|.
    \end{align*}
    Owing to the estimates \textrm{(i)}--\textrm{(v)}, we get for $r\ge 1$
    \begin{align*}
        |\eta(r)| \le C \big\{ |\eta_b| + u_b^2 \rho_{+}^{-\gamma} \big\} r^{-2(n-1)}  + \rho_{+}  C(\rho_{+}) \big\{ u_b^2 + |\eta_b| \big\} r^{-2(n-1)} \sup\limits_{1\le s \le r} |s^{2(n-1)} \eta(s)|.
    \end{align*}
    Let $R>1$. Multiplying both sides of the above inequality by $r^{2(n-1)}$ with $r\in[1,R]$. Then taking the supremum over $1\le r \le R$, we obtain that there exists $\ep=\ep(\rho_+,\mu,\gamma,K,n)>0$ and $C=C(\mu,\gamma,K,n)>0$ for which if $u_b + |\eta_b|\le \ep$ then
    \begin{equation}\label{temp:first}
        \sup\limits_{1\le r \le R} r^{2(n-1)} |\eta(r)| \le C \big\{ |\eta_b| + \rho_{+}^{-\gamma} u_b^2  \big\} \qquad \text{ for an arbitrary } \ R >1.
    \end{equation}
    Taking the limit $R\to \infty$, we obtain the first inequality for \eqref{etaEst1}.
    
    Next, we derive the estimate of $\eta_r$. Applying integration by parts on the formula \eqref{etae}, we obtain
    \begin{align}
        \eta_r(r) %=& - \eta_b \frac{n\kappa}{m_b} r^{n-1} e^{-\frac{\kappa}{m_b}(r^n-1)} + F[\eta](r) -\frac{n\kappa}{m_b} r^{n-1} \int_{1}^r e^{\frac{\kappa}{m_b}(s^n-r^n)} F[\eta](s)\, ds\nonumber\\
        =& \underbrace{\frac{(1-n)v_{+}}{n\mu \kappa} m_b^2 r^{-2n+1}}_{\textrm{(a)}} + \underbrace{\Big\{ F[\eta](1) - \frac{n\kappa \eta_b}{m_b} - \frac{(n-1)v_{+}}{n\mu\kappa}m_b^2 \Big\} r^{n-1} e^{-\frac{\kappa}{m_b}(r^n-1)}}_{\textrm{(b)}}\nonumber\\
        &+ \underbrace{\frac{1-n}{\mu} r^{n-1} \int_{1}^{r} \Big\{ \frac{(3n-2)v_{+}}{n\kappa} \frac{m_b^2}{s^{3n-1}} + \frac{m_b \eta(s)}{s^{2n-1}} \Big\} e^{\frac{\kappa}{m_b}(s^n-r^n)}\, ds}_{\textrm{(c)}}  \nonumber\\
        & + \underbrace{\frac{r^{n-1}}{\mu} \int_{1}^r\!\! \eta_r(s) \Big\{ \frac{m_b}{s^{2(n-1)}} + \frac{N^{\prime}\big(\eta(s)\big)}{m_b} \Big\}  e^{\frac{\kappa}{m_b}(s^n-r^n)}  \,ds}_{\textrm{(d)}}. \label{er-long}
    \end{align}
    By \eqref{etae}, \eqref{E0} and the assumption \eqref{urg}, it follows that
    \begin{align*}
        |\textrm{(a)}| \le C \frac{\rho_{+}^{-1}}{\rho_{+}^{\gamma+1}} \rho_{+}^2 u_b^2 r^{-2n+1} 
        \le C \rho_{+}^{-\gamma} u_b^2 r^{-2n+1} \le C u_b r^{-2n+1}.
    \end{align*}
    Using the definition of $F[\eta]$, \eqref{NMVT} and \eqref{temp:first}, we verify that
    \begin{align*}
        &\Big|F[\eta](1)-\frac{n\kappa \eta_b}{m_b}\Big|=\bigg|\frac{m_b v_{+}}{2\mu}+\frac{m_b \eta_b}{\mu}-\frac{(n-1)m_b}{\mu}\int_{1}^{\infty}\frac{\eta(s)}{s^{2n-1}}\,ds + \frac{p(v_{+}+\eta_b)-p(v_{+})}{\mu m_b}\bigg|\\
        %\le & C \big\{ u_b + \rho_{+} u_b |\eta_b|  + m_b \big(|\eta_b|+u_b^2\rho_{+}^{-\gamma}\big) \int_{1}^\infty s^{-4n+3}\, ds + \frac{|\eta_b|}{\rho_{+} u_b} \sup\limits_{v_{+}/2\le \xi\le 2v_{+}}p^{\prime}(\xi) \big\} \\
        %\le & C \big\{ u_b + \rho_{+} u_b |\eta_b| + \rho_{+} u_b \big(|\eta_b| + u_b^2 \rho_{+}^{-\gamma}\big)  + \frac{|\eta_b|}{u_b} \rho_{+}^{\gamma} \big\}\\
        \le & C \Big\{ u_b + \rho_{+} u_b \big(|\eta_b| + u_b^2 \rho_{+}^{-\gamma}\big)  + \frac{|\eta_b|}{u_b} \rho_{+}^{\gamma} \Big\}.
    \end{align*}
    By the property of exponential function, the condition $m_b\le n\kappa$ and \eqref{urg}, we have
    \begin{align*}
        |\textrm{(b)}| \le& \Big| F[\eta](1) - \frac{n\kappa \eta_b}{m_b} - \frac{(n-1)v_{+}}{n\mu\kappa}m_b^2 \Big| r^{-2n+1} (r^n)^{\frac{3n-2}{n}} e^{-r^n+1}\\ 
        %\le& C \Big\{ \rho_{+}^{\gamma} \frac{|\eta_b|}{u_b} + u_b + \rho_{+}^{-\gamma} u_b^2  + \rho_{+}|\eta_b| u_b + \rho_{+}^{1-\gamma}u_b^3 \Big\} r^{-2n+1}\\
        \le& C \Big\{ \rho_{+}^{\gamma} \frac{|\eta_b|}{u_b} + u_b  + \rho_{+}|\eta_b| u_b + \rho_{+} u_b^2 \Big\} r^{-2n+1},
    \end{align*}
    Applying Proposition \ref{prop:H-M}, \eqref{E0} and \eqref{urg}, we obtain
    \begin{align*}
        |\textrm{(c)}|%=& \frac{n-1}{\mu} r^{-2n+1} \bigg| r^{3n-2} \int_{1}^{r} \Big\{ \frac{(3n-2)v_+}{n\kappa} \frac{m_b^2}{s^{3n-1}} + \frac{m_b \eta(s)}{s^{2n-1}} \Big\} e^{\frac{\kappa}{m_b}(s^n-r^n)}\, ds \bigg| \\
        \le& \frac{C}{r^{2n-1}} \frac{u_b}{\rho_{+}^{\gamma}} \sup\limits_{r\ge 1} \Big| \frac{(3n-2)v_{+}}{n\kappa} \frac{m_b^2}{r} + m_b r^{n-1} \eta(r) \Big| 
        %\\ \le& C r^{-2n+1} \rho_{+}^{-\gamma} u_b \big\{ \rho_{+}^{-\gamma} u_b^2 + \rho_{+} u_b v_{+} \big\} \\ \le& C \big\{ \rho_{+}^{-\gamma} u_b^2 + \rho_{+}^{-2\gamma} u_b^3 \big\} r^{-2n+1} 
        \le C u_b r^{-2n+1}.
    \end{align*}
    By mean value theorem, we have $|N^{\prime}\big(\eta(s)\big)| \le C \rho_{+}^{\gamma+2} |\eta(s)|$. Using this, \eqref{temp:first} and Proposition \ref{prop:H-M}, we get
    \begin{align*}
        |\textrm{(d)}|%=& \frac{r^{-2n+1}}{\mu}\bigg| r^{3n-2} \int_{1}^{r} \!\! s^{2n-1}\eta_{r}(s) \Big\{ \frac{m_b}{s^{4n-3}} + \frac{N^{\prime}\big(\eta(s)\big)}{m_b s^{2n-1}} \Big\} e^{\frac{\kappa}{m_b}(s^n-r^n)}\, ds \bigg|  \\
        %\le& C r^{-2n+1} \sup\limits_{1\le \tau\le r}\big\{\tau^{2n-1}\eta_r(\tau)\big\} \cdot \bigg| r^{3n-2} \int_{1}^{r} \!\! \Big\{ \frac{m_b}{s^{4n-3}} + \rho_{+}^{\gamma+2}\frac{|\eta(s)|}{m_b s^{2n-1}} \Big\} e^{\frac{\kappa}{m_b}(s^n-r^n)}\, ds \bigg|\\
        \le& C r^{-2n+1} \sup\limits_{1\le \tau\le r}\big\{\tau^{2n-1}\eta_r(\tau)\big\} \cdot \rho_{+}^{-\gamma} u_b \sup\limits_{s\ge 1}\bigg\{ \frac{\rho_{+} u_b}{s^{n-1}} + \rho_{+}^{\gamma+1}\frac{s^{n-1}|\eta(s)|}{u_b}  \bigg\}\\
        %\le &  C r^{-2n+1} \sup\limits_{1\le s \le r} \big\{ s^{2n-1} \eta_r(s) \big\} \big\{ \rho_{+}^{1-\gamma} u_b^2 + \rho_{+} ( |\eta_b| + \rho_{+}^{-\gamma} u_b^2 ) \big\} \\
        \le& C \big\{ \rho_{+}|\eta_b| + \rho_{+}^{1-\gamma} u_b^2 \big\} r^{-2n+1} \sup\limits_{1\le s \le r} \big\{ s^{2n-1} \eta_r(s) \big\}.
    \end{align*}
    Substituting the estimates for \textrm{(a)}--\textrm{(d)} in \eqref{er-long}, multiplying both sides by $r^{2n-1}$ and taking supremum over $r\in [1,R]$ for $R>1$, we get
    \begin{align*}
        \sup\limits_{1 \le r\le R} r^{2n-1} |\eta_r (r)| \le& C \Big\{ \rho_{+}^{\gamma} \frac{|\eta_b|}{u_b} + u_b  + \rho_{+}|\eta_b| u_b + \rho_{+} u_b^2 \Big\}\\ 
        &+ C \big\{\rho_{+}|\eta_b| + \rho_{+}^{1-\gamma} u_b^2 \big\} \sup\limits_{1\le r \le R} r^{2n-1} |\eta_r(r)|.
    \end{align*}
    Thus, there exist $\ep=\ep(\rho_+,\mu,\gamma,K,n)>0$ and $C=C(\mu,\gamma,K,n)>0$ such that if $|u_b|+|\eta_b|\le \ep$ then for arbitrary $R>1$,
    \begin{align*}
        \sup\limits_{1\le r \le R} r^{2n-1} |\eta_r (r)| %\le C \Big\{ \rho_{+}^{\gamma} \frac{|\eta_b|}{u_b} + u_b  + \rho_{+}|\eta_b| u_b + \rho_{+} u_b^2 \Big\}
        \le C \Big\{ u_b + \rho_{+}^{\gamma} \frac{|\eta_b|}{u_b} \Big\}.
    \end{align*}
    Taking $R\to \infty$ yields the second inequality of \eqref{etaEst1}. %Finally, we remark that the estimate for $\eta_{rr}$ in \eqref{etaEst2} is obtained in the same way by taking derivative $\d_r$ on \eqref{er-long} and integrating by parts to express $\eta_{rr}$ as a representation formula. The procedure for this is skipped since it is rather technical and it repeats the arguments used for the derivations of \eqref{etaEst1}.
\end{proof}

The decay rate estimates \eqref{rho} and \eqref{ur}  are obtained by combining Lemma \ref{lemma:etaEst} with the definition $\eta=\rt^{-1}- \rho_{+}^{-1}$ and the formula \eqref{stu}.

\subsection{Classification of extremum points in density}\label{ssec:extremum}
This subsection is devoted to the classification of extremum points in $\rt(r)$, which is stated in \ref{item:I}--\ref{item:II} of Lemma \ref{le:st}. To do this, it suffices to classify the extremum points of $\eta(r)= 1/\rt(r) - 1/\rho_+$ instead. First, we prove the following proposition.
\begin{proposition}\label{prop:etalr}
    Let $\eta(r)$ be the solution constructed in Section \ref{ssec:stExist}. Then there exists a constant $R=R(\rho_+,\mu,\gamma,K,n)>1$ such that for $r\ge R$,
    \begin{equation}\label{etalr}
        C^{-1} u_b^2 r^{-2(n-1)} \le \eta(r) \quad \text{ and } \quad \eta_r(r)<0.
    \end{equation}
\end{proposition}
\begin{proof}
    We only show the case for $\eta_b<0$ since the other case $\eta_b\ge 0$ is handled similarly. From the representation formula \eqref{etae}, we have
    \begin{align*}
        \eta(r) =& \underbrace{\frac{m_b}{\mu}\int_{1}^r \!\!\big( \frac{v_+}{2} + \eta(s) \big) s^{1-n} e^{\frac{\kappa}{m_b}(s^n-r^n)}\, ds}_{\textrm{(i)}} + \underbrace{\eta_b e^{-\frac{\kappa}{m_b}(r^n-1)}}_{\textrm{(ii)}}\\
        &+ \underbrace{\int_{1}^{r}\frac{s^{n-1}}{\mu m_b}N\big(\eta(s)\big) e^{\frac{\kappa}{m_b}(s^n-r^n)} \,ds}_{\textrm{(iii)}} + \underbrace{\frac{(1-n)m_b}{\mu} \int_{1}^r\int_{s}^{\infty} \frac{\eta(\tau)}{\tau^{2n-1}} e^{\frac{\kappa}{m_b}(s^n-r^n)}\, d\tau ds}_{\textrm{(iv)}}.
    \end{align*}
    By \eqref{E0}, it holds that $\frac{1}{2}v_{+} + \eta(s) \ge \frac{1}{4}v_{+}$. Moreover, by the decay property of exponential function, there exists a constant $R_1 > 1$ such that the inequality $1-e^{\frac{\kappa}{m_b}(1-r^n)} \ge \frac{1}{2}$ holds for $r\ge R_1$. Using these, we get
    \begin{align*}
        \textrm{(i)} \ge& \frac{v_{+}m_b}{4\mu} \int_{1}^{r} s^{1-n} e^{\frac{\kappa}{m_b}(s^n-r^n)}\, ds  %\\ =& e^{-\frac{\kappa}{m_b}r^n} \frac{m_b v_{+}}{4\mu} \int_{1}^r \frac{m_b}{n\kappa} s^{2(1-n)} \frac{\kappa}{m_b} n s^{n-1} e^{\frac{\kappa}{m_b}s^n}\, ds\\ &
        \ge \frac{m_b^2 v_{+}}{4n\kappa \mu} r^{2(1-n)} e^{-\frac{\kappa}{m_b}r^n} \int_{1}^r \frac{\kappa}{m_b} n s^{n-1} e^{\frac{\kappa}{m_b}s^n}\, ds\\
        =& \frac{m_b^2 v_{+}}{4n\kappa \mu} r^{2(1-n)} \big\{1- e^{\frac{\kappa}{m_b}(1-r^n)}\big\} \ge \frac{v_{+}}{8n\kappa \mu} m_b^2 r^{2(1-n)}, \quad\qquad \text{for } \ r\ge R_1.
    \end{align*}
    Next, for exponential function, we have $- e^{-\frac{\kappa}{m_b}r^n} \ge - C \kappa^{-4} m_b^4 r^{-4n}$. Since $\eta_b<0$, this inequality implies that there exists another constant $R_2>1$ for which if $r\ge R_2$, then
    \begin{align*}
        \textrm{(ii)} \ge \eta_b e^{\frac{\kappa}{m_b}} C \frac{m_b^4}{\kappa^4} r^{-4n} = 
        \eta_b e^{\frac{\kappa}{m_b}} C \frac{m_b^4}{\kappa^4} r^{-2(n+1)} r^{2(1-n)} \ge - \frac{v_{+}}{32n\kappa \mu} m_b^2 r^{2(1-n)}. 
    \end{align*}
    By Taylor's theorem, $ N\big(\eta(s)\big) = \frac{1}{2} p^{\prime\prime}(\xi) |\eta(s)|^2 >0$ for some $\xi\in (\frac{3}{4}v_{+}, \frac{5}{4}v_{+})$. Thus it holds that $\textrm{(iii)} >0$. Finally, by Lemma \ref{lem:H-M21}, we have
    \begin{align*}
        \bigg|\int_{s}^{\infty} \!\! \frac{\eta(\tau)}{\tau^{2n-1}}\, d\tau \bigg| \le C \big\{ u_b^2 + |\eta_b| \big\} \int_{s}^{\infty} \!\! \tau^{-3n+2} \, d \tau = C \big\{ u_b^2 + |\eta_b| \big\}  s^{3(1-n)}.
    \end{align*}
    Using this inequality in \textrm{(iv)} and applying Proposition \ref{prop:H-M}, we obtain that there exists a constant $R_3\ge 1$ such that for $r\ge R_3$,
    \begin{align*}
        |\textrm{(iv)}| \le& C \big\{ u_b^2 + |\eta_b| \big\} m_b r^{3(1-n)} \cdot r^{3(n-1)} \int_{1}^{r}\!\! s^{3(1-n)}e^{\frac{\kappa}{m_b}(s^n-r^n)}\, ds\\
        \le & C \big\{ u_b^2 + |\eta_b| \big\} m_b^2 r^{3(1-n)} \le \frac{v_{+}}{32 n \kappa \mu} m_b^2 r^{2(n-1)}.
    \end{align*}
    Substituting the estimates for \textrm{(i)}--\textrm{(iv)} in the formula \eqref{etae}, we obtain that
    \begin{align*}
        \eta(r) \ge \textrm{(i)} + \textrm{(ii)} - |\textrm{(iv)}| %\ge \Big\{ \dfrac{v_{+}}{8n\kappa \mu} - \frac{v_{+}}{32 n\kappa \mu} - \frac{v_{+}}{32 n\kappa \mu} \Big\} m_b^2 r^{2(1-n)} = \frac{v_{+}}{16 n \kappa \mu} r^{2(1-n)}
        \ge  C^{-1} m_b^2 r^{2(1-n)},
    \end{align*}
    for $r\ge R \vcentcolon= \max\{R_1,R_2,R_3\}> 1$. This proves the first inequality of \eqref{etalr}. By the same procedure applied on the formula of $\eta_r(r)$ given in \eqref{er-long}, it is also verified that there exists a constant $R>1$ such that $\eta_r(r) <0$ for $r\ge R$. The detailed verification for this is abbreviated since it repeats the same argument.
\end{proof}

The sign of the boundary data $\eta_b=v_b-v_+$ determines the extremum points of $\eta(r)$ as follows
\begin{lemma}
    There exists a constant $\ep=\ep(\rho_+,\mu,\gamma,K,n)>0$ such that if $|\eta_b|+ u_b \le \ep$, then the problem \eqref{eta0} has a unique solution $\eta$, for which the following holds
    \begin{enumerate}[label=\textnormal{(\Alph*)},ref=\textnormal{(\Alph*)}]
        \item\label{item:A} if $\eta_b\le 0$, then there exists a unique point $r^{\ast}>1$ such that $\eta(r)$ is strictly increasing in $1\le r<r_{\ast}$ and strictly decreasing in $r>r_{\ast}$. Moreover,
        \begin{equation}
            \max_{r\ge 1}\eta(r) = \eta(r^{\ast}) >0, \qquad \lim\limits_{r\to\infty} \eta(r) = 0. 
        \end{equation}
        \item\label{item:B} if $\eta_b>0$, then $\eta(r)$ is strictly decreasing in $r\ge 1$.
    \end{enumerate}
\end{lemma}
\begin{proof}
    We claim that for $r>1$,
    \begin{equation}\label{noMin}
        \text{either } \quad  \eta_r(r) \neq 0 \quad \text{ or } \quad \eta_{rr}(r) < 0. 
    \end{equation}
    We prove the above statement by contradiction. Suppose otherwise that there exists $\bar{r}> 1$ for which both $\eta_r(\bar{r}) = 0$ and $ \eta_{rr}(\bar{r})\ge 0$ hold. Then substituting $\eta_r(\bar{r})=0$ in \eqref{etae2}, we get
    \begin{align} \label{Fbarr}
        \frac{n\kappa}{m_b} \bar{r}^{n-1} \eta(\bar{r}) = F[\eta](\bar{r}). 
    \end{align}
    Differentiating \eqref{etae2} and taking $r=\bar{r}$ on the resultant equality, we get
    \begin{align}
        \nonumber \eta_{rr}(\bar{r}) %=& - \frac{n\kappa }{m_b} \bar{r}^{n-1} \d_r \eta(\bar{r}) - \frac{n(n-1)\kappa}{m_b} \bar{r}^{n-2} \eta(\bar{r})  + \d_r F[\eta](\bar{r})\\
        =& \Big\{\frac{m_b}{\mu \bar{r}^{n-1}}+ \frac{\bar{r}^{n-1}N^{\prime}\big(\eta(\bar{r})\big)}{\mu m_b} - \frac{n\kappa \bar{r}^{n-1}}{m_b}\Big\}\eta_r(\bar{r}) +\frac{n-1}{\bar{r}} F[\eta](\bar{r}) \\
        & + (n-1) \Big\{\frac{m_b}{\mu \bar{r}^n} - \frac{n\kappa}{m_b} \bar{r}^{n-2}\Big\}\eta(\bar{r}) - \frac{2(n-1)m_b}{\mu \bar{r}^n}\Big( \dfrac{v_{+}}{2} + \eta(\bar{r}) \Big).\label{etarr}
    \end{align}
    Substituting \eqref{Fbarr}, $\eta_r(\bar{r})=0$ and $\eta_{rr}(\bar{r})\ge 0$ in \eqref{etarr}, we obtain that
    \begin{align*}
        0 \le \eta_{rr}(\bar{r}) = -\frac{(n-1)m_b}{\mu} \frac{\eta(\bar{r})}{\bar{r}^n} - \frac{(n-1)m_b}{\mu} \frac{v_+}{\bar{r}^n}.
    \end{align*}
    Since $m_b>0$, it follows that $\eta(\bar{r})\le - v_{+}$ which implies $v(\bar{r})\le 0$. However this is impossible since $v>0$ by \eqref{E0}. Therefore a local minimum or a point of inflection does not exist for $r>1$ and thus the only possible such point occurs at $r=1$. 
    
    From here, the proof is divided into $2$ cases as follows
    \paragraph{\bf Case 1: $\eta_b \le 0$.} Proposition \ref{prop:etalr} implies there exists a point $R>1$ such that we have $\eta(r)\ge C^{-1} m_b^2 r^{-2(n-1)}>0$ for $r\ge R$. Since $\eta\in \mathcal{B}^2[1,\infty)$, $\eta(r)\to 0$ as $r\to \infty$ and \eqref{noMin} holds, applying Rolle's theorem yields that there exits a unique point $r_{\ast}>1$ with $\eta(r_{\ast})=\text{max}_{r\ge 1} \eta(r)$.    
    \paragraph{\bf Case 2: $\eta_b> 0$.} We consider two sub-cases $\eta_r\vert_{r=1}\le 0$ or $\eta_r \vert_{r=1}>0$. %First if $\eta_r \vert_{r=1}>0$, then Proposition \ref{prop:etalr}, \eqref{noMin} and Intermediate Value theorem imply that there is a unique point $r_{\ast}>1$, at which $\eta(r)$ achieves global maximum. 
First, consider the case $\eta_r\vert_{r=1}\le 0$. We show $\eta_r \vert_{r=1}<0$ by contradiction. Suppose otherwise that $\eta_r \vert_{r=1} =0$. Then by repeating the same argument used in the proof of \eqref{noMin}, we get that $\eta_b=\eta\vert_{r=1}\le -v_{+}$. This leads to $v_b\le 0$, which is a contradiction. Thus, for this sub-case, we must have $\eta_r \vert_{r=1} < 0$. Suppose $\eta(r)$ is not strictly decreasing, which means there exist two points $1 < r_1 < r_2 $ such that $\eta_r \ge 0$ in $r\in (r_1,r_2)$. Then by intermediate value theorem, there exists a point $\bar{r}>1$ such that $\eta_r(\bar{r})=0$ and $\eta_{rr}(\bar{r})\ge 0$. This is a contradiction to \eqref{noMin}. Therefore we conclude that $\eta(r)$ is strictly decreasing. Finally we observe from the formula \eqref{etae2} and Lemma \ref{lemma:etaEst} that, we can choose a small constant $\ep>0$ for which if $\eta_b> 0$ and $u_b+|\eta_b|\le \ep$ then $\eta_r \vert_{r=1} \le 0$. This rules out the second sub-case $\eta_r \vert_{r=1}>0$ and we conclude that $r\mapsto \eta(r)$ is strictly decreasing with $\eta_b> 0$ and $u_b+|\eta_b|\le \ep$ for $\ep$ small enough. 
\end{proof}

\section{A-priori estimates in Lagrangian coordinate}\label{sec:prioriE}
The main strategy for the proof of Theorem~\ref{mth} is to employ the energy method and derive a-priori estimates in the Lagrangian coordinate. For inflow problem formulated in the Eulerian coordinate, the main difficulty is the boundary term in the estimate of $\psi\vcentcolon=u-\ut$. To resolve this issue, we use the structure of the equations under the Lagrangian formulation. First we state the local-in-time existence in the Eulerian coordinates
\begin{lemma}\label{lm:local}
    Suppose the initial data $(\rho_0,u_0)$ satisfies the same conditions stated in Theorem \ref{mth}. Then there exists a constant $T_0>0$, depending only on $\| r^{\frac{n-1}{2}} (\rho_0-\rt) \|_{\mathcal{B}^{1+\sigma}}$ and $\| r^{\frac{n-1}{2}} (u_0-\ut) \|_{\mathcal{B}^{2+\sigma}}$, such that the initial-boundary value problem \eqref{nse}--\eqref{compa} and \eqref{u+} has a unique solution satisfying 
    \begin{gather*}
        r^{\frac{n-1}{2}} (\rho-\rt ), \qu r^{\frac{n-1}{2}} u, \qu r^{\frac{n-1}{2}} (\rho-\rt)_r\,, \qu r^{\frac{n-1}{2}} u_r \ \in C\big([0,T]; L^2(1,\infty)\big),\\
        \rho \in \mathcal{B}^{1+\frac{\sigma}{2},\, 1+\sigma}\big([0,T]\times[1,\infty)\big), \qquad u\in \mathcal{B}^{1+\frac{\sigma}{2},\, 2+\sigma}\big([0,T]\times[1,\infty)\big),
    \end{gather*}
    for an arbitrary $T\in (0,T_0)$.
\end{lemma}

The energy norm for the a-priori estimate in Eulerian coordinate is defined by
\begin{gather}
    N_{\textrm{E}}(t) := \sup\limits_{0\le s \le t} \Big\{ \big\| r^{\frac{n-1}{2}}(\rho-\rt,u-\ut) \big\| + \big\| r^{\frac{n-1}{2}}(\rho-\rt)_r \big\| + \big\| r^{\frac{n-1}{2}}(u-\ut)_r \big\| \Big\},\label{NE}\\
    \begin{aligned}
        M_{\textrm{E}}^2(t) :=& \int_{0}^{t}\!\!  \big\| r^{\frac{n-3}{2}}(\rho-\rt)_r \big\|^2 + \big\| r^{\frac{n-1}{2}}(u-\ut)_r \big\|^2 + \big\| r^{\frac{n-3}{2}} (u-\ut)_{rr} \big\|^2\,  d\tau \\ &+ \int_{0}^{t} u_b\big|(u-\ut)_r(\tau,1)\big|^2 + u_b^3\big|(\rho-\rt)_r(\tau,1)\big|^2\, d\tau.
    \end{aligned}\label{ME}
\end{gather}
By Lemma \ref{lm:local}, for a given $\epsilon>0$ there exists a time $T(\epsilon)>0$ such that if $N_{E}(0) \le \epsilon$, then $N_{E}(t)\le 2\epsilon$ for all $t\in[0,T(\epsilon)]$. Combining this argument with the Sobolev embedding theorem and Lemma \ref{le:st}, we find constants $\ep_0=\ep_0(\rho_+,\mu,\gamma,K,n)>0$ and $T=T(\ep_0)>0$ such that if $|\rho_b-\rho_+| + u_b + N_{E}(0) \le \ep_0$ then 
\begin{equation}\label{priori-Rho}
    N_{E}(t)\le 2\ep_0, \qquad \tfrac{1}{2}\rho_{+} \le \rho(r,t) \le 2 \rho_+ \qquad \text{for all } \ (r,t)\in [1,\infty)\times[0,T].
\end{equation}
\subsection{Reformulation in Lagrangian Coordinate}\label{ssec:reform}
Let $(\rho,u)(r,t)$ be the local-in-time solution with the maximum time of existence $T>0$. Consider the transformation $(r,t)\mapsto X(r,t)\vcentcolon [1,\infty)\times[0,T]\to \R$ defined by
\begin{equation}\label{RInv}
	X(r,t) \vcentcolon= -m_b t + \int_{1}^{r} \rho(y,t) y^{n-1} d y \quad \text{ where } \quad m_b\vcentcolon= \rho_b u_b>0.
\end{equation}
Since $\frac{1}{2}\rho_{+} \le \rho(r,t)\le 2\rho_+$ in $(r,t)\in[1,\infty)\times[0,T]$, Inverse Function theorem implies that, for each $t\in[0,T]$ there exists a unique map $x\mapsto R(x,t) \vcentcolon [-m_b t, \infty)\to [1,\infty)$ such that $X(R(x,t),t)=x$. Namely,
\begin{subequations}\label{Rexist}
\begin{gather}
    x= -m_b t + \int_{1}^{R(x,t)} \rho(r,t) r^{n-1} d r \quad \text{for each $-m_b t \le  x$,}\label{Rmass}\\
     R(-m_b t,t)=1 \quad \text{ and } \quad R(x,t)\to \infty \ \text{ as } \ x\to\infty.\label{Rbdry}
\end{gather} 
\end{subequations}
Using the above construction, we set the Lagrangian space-time domain $\mQ(T)$ and the Lagrangian snapshot domain $S(t)$ at time $t\in[0,T]$ to be
\begin{subequations}\label{Ldomain}
\begin{align}
\mQ(T)\vcentcolon=& \{ (x,t)\in\R\times[0,T] \;|\; -m_b t \le x < \infty \}, \label{lsp} \\
S(t) \vcentcolon=& \{ x\in\R \;|\; -m_b t \le x < \infty \} \qu \text{for $t\in[0,T]$.} \label{lt}
\end{align}
\end{subequations}
By the Implicit Function theorem and the regularity of $(\rho,u)$, the derivatives $R_x$ and $R_t$ exist and it follows from \eqref{nse1} and \eqref{Rexist} that for an arbitrary $(x,t)\in\mQ(T)$
\begin{equation}\label{Rdiff}
	R_t(x,t) = u(R(x,t),t), \qu R_x(x,t) = \dfrac{R(x,t)^{1-n}}{\rho(R(x,t),t)}.
\end{equation}
Therefore $R(x,t)$ satisfies the following identity
\begin{equation}\label{R}
R(x,t)= R_0(x) + \int_{0}^{t} u(R(x,s),s)d s = \Big(1+n\int_{-m_b t}^{x} \dfrac{1}{\rho(R(y,t),t)}\,d y\Big)^{\frac{1}{n}},
\end{equation} 
where $R_0(x)$ is implicitly defined by 
\begin{equation*}
x = \int_{1}^{R_0(x)}\!\! \rho_0 (r) r^{n-1}\, d r \ \text{ for $0 \le x <\infty$.}
\end{equation*}
The transformation from the Eulerian coordinate $(r,t)$ to the Lagrangian coordinate $(x,t)$ is executed by the equation $r=R(x,t)$. Let $v \vcentcolon= 1/\rho$ be the specific volume. Using \eqref{Rdiff}, we obtain equations for $(\hat{u}, \hat{v})(x,t):=(u,v)(R(x,t), t)$ from \eqref{nse} as
\begin{subequations}
\label{nsl-hat}
\begin{gather}
\hv_t = (r^{n-1} \hu)_x, \label{nsl-hat1} \\
\hu_t = \mu r^{n-1} \Big( \frac{(r^{n-1} \hu)_x}{\hv} \Big)_x
- r^{n-1} p(\hv)_x, \label{nsl-hat2}
\end{gather}
\end{subequations}
where $p(v)=K v^{-\gamma}$ and $r=R(x,t)$. The initial and the boundary conditions
for $(\hv,\hu)$ are derived from \eqref{ub} as
\begin{subequations}\label{lc}
\begin{gather}
\hv(x,0) = \hv_0(x) := 1/\rho_0(R_0(x)), \qu \hu(x,0) = \hu_0(x) := u_0(R_0(x)), \label{lic-hat} \\
\hv_0(0)= v_b \equiv 1/ \rho_b, \qu \hu_0(0)=u_b, \label{lcompa-hat}\\
\hu(-m_b t,t) = u_b, \qu \hv(-m_b t,t)= v_b \equiv 1/\rho_b. \label{lbc-hat}
\end{gather}
\end{subequations}
%The compatibility condition \eqref{compa} in Lagrangian coordinates is given by\todo{Change it to inflow too}
%\begin{equation}\label{lcompa1-hat} (\hv_0,\hu_0)\vert_{x=0}=(v_b,u_b), \qu \Big\{ \rho_b u_b (\hu_0)_x + \mu \Big( \dfrac{(\hu_0)_x}{\hv_0} \Big)_x - p(\hv_0)_x \Big\} \Big\vert_{x=0} = 0. \end{equation}
Since the spatial variable in Eulerian coordinate, $r = R(x,t)$, is a function of spatial and temporal variables in the Lagrangian coordinate, the stationary solution $(\rt,\ut)$ also depends on $(x,t)$. For simplicity we abuse the notation to denote
\begin{gather}
    \rt(x,t):= \rt(R(x,t)), \qu \vt(x,t) := 1/\rt(R(x,t)), \qu \ut(x,t) := \ut(R(x,t)),\nonumber\\
    \rt_r(x,t) := \rt_r(R(x,t)), \qu \vt_r(x,t) := -\big(\frac{\rt_r}{\rt^2}\big)(R(x,t)), \qu \ut_r(x,t) := \ut_r(R(x,t)),\nonumber\\
    \tilde{\mL}(x,t)\vcentcolon= \mu \d_r\Big( \frac{\d_r(r^{n-1}\ut)}{r^{n-1}} \Big)\Big\vert_{r=R(x,t)} = \big\{ \rt \ut \d_r \ut + \d_r P(\rt) \big\}\big\vert_{r=R(x,t)}.\label{mL}
\end{gather}
With these constructions, we define the difference functions as 
\begin{equation*}
	\phi(x,t)\vcentcolon= \hv(x,t)-\vt(x,t), \qu \psi(x,t)\vcentcolon= \hu(x,t)-\ut(x,t).
\end{equation*}
Then $(\phi,\psi)$ satisfies the following equations in domain $(x,t)\in \mQ(T)$
\begin{subequations}\label{dLag}
    \begin{gather}
        \phi_t - (r^{n-1}\psi)_x = \mR_1,\label{dLag-1} \\
        \psi_t - \mu r^{n-1}\Big(\dfrac{(r^{n-1}\psi)_x}{v}\Big)_x + r^{n-1} p^{\prime}(v) \phi_x = \mR_2,\label{dLag-2}
    \end{gather}
    where
    \begin{gather}
    	\mR_1 := \dfrac{\vt_r \ut}{\vt} \phi - \vt_r \psi, \qu \mR_2 := \dfrac{\ut_r \ut}{\vt} \phi - \vt_r \dfrac{p^{\prime}(v)-p^{\prime}(\vt)}{v-\vt} v \phi - \ut_r \psi.
    \end{gather}
\end{subequations}
Moreover, boundary condition \eqref{ub} and \eqref{Rbdry} yield that
\begin{equation}\label{ldbc}
	\phi(-m_b t,t)=0, \qu \psi(-m_b t,t)=0, \qu \text{for } \ t\in[0,T].
\end{equation}
In what follows, we omit the hat ``\,$\hat{ \ }$\,'' to express 
functions in the Lagrangian coordinate simply as $(v,u)\equiv (\hv,\hu)$.

Under this formulation, we define the energy norm in Lagrangian coordinate as
\begin{gather}
    N(t) := \sup\limits_{0\le \tau \le t} \big\{ \| (\phi,\psi) \|_{L^{2}(S(\tau))} + \| r^{n-1} (\phi_x, \psi_x) \|_{L^{2}(S(\tau))} \big\},\label{Nt}\\
    \begin{aligned}
        M^2(t) :=& \iint_{\mQ(t)} \Big\{ \frac{\psi^2}{r^2} + r^{2(n-1)} \psi_x^2 + r^{2n-4} \phi_x^2 + r^{4n-6} \psi_{xx}^2 \Big\} \,dx d\tau \\
        &+ \int_{0}^{T} \big\{ u_b \psi_x^2 + u_b^3 \phi_x^2 \big\}(-m_b \tau, \tau) \, d\tau. 
    \end{aligned}\label{Mt}
\end{gather}
Let $T>0$ be the time obtained in the inequality \eqref{priori-Rho}. Using the definition \eqref{NE}--\eqref{ME} and the differential relations \eqref{Rdiff}, it is verified that there exists some constant $C(\rho_+)=C(\rho_+,\mu,\gamma,K,n)>0$ such that for $t\in[0,T]$,
\begin{equation} \label{NEL}
    C(\rho_{+})^{-1} N_{E} \le N \le C(\rho_{+}) N_{E}, \qu C(\rho_{+})^{-1} M_{E}  \le M \le C(\rho_{+}) M_{E}.
\end{equation}
Consequently, there are constants $(\ep_0,T)$ such that if $|\rho_b-\rho_+|+u_b + N_{E}(0) \le \ep_0$, then 
\begin{equation}\label{priori-Lag}
    N(t)\le 2\ep_0, \qquad \frac{1}{2\rho_{+}} \le v(x,t) \le \dfrac{2}{\rho_+} \quad \text{for } \ (x,t)\in \mQ(T).
\end{equation}

The main estimate of this section is stated in the following theorem.
\begin{theorem}\label{thm:Lag}
    Suppose $(\phi,\psi)$ solves \eqref{dLag} in $\mQ(T)$ for $T>0$ such that \eqref{priori-Lag} holds. Then there exist constants $\alpha=\alpha(\mu,\gamma,K,n)>0$ which is independent of $\rho_+$ and $\ep=\ep(\rho_+,\mu,\gamma,K,n)>0$, such that if $\rho_+\in (0,\alpha]$, $|\rho_b-\rho_+|\le u_b^2$ and $u_b< \ep$, then
    \begin{align*}
        \sup\limits_{0\le t\le T} N^2(t) + M^2(T) \le C(\rho_+) \Big\{ 1 + \frac{1}{u_b^2} \Big\} N^2(0),
    \end{align*}
    where $C(\rho_+)>0$ is a constant which depends only on $\rho_+$, $\mu$, $\gamma$, $K$ and $n$.
\end{theorem}

\subsection{Relative energy estimate}\label{ssec:REE}
Define the relative energy as
\begin{equation}
    \mE := \dfrac{|u-\ut|^2}{2} + G[v,\vt], \qu \text{where } \ G[v,\vt]:= \int_{1/\vt}^{1/v} \dfrac{P(z)-P(\vt^{-1})}{z^2} \,d z.
\end{equation}
For the isentropic pressure law $P(\rho)=K \rho^{\gamma}$, we have
\begin{equation*}
    G[v,\vt] = \left\{\begin{aligned}
        &\tfrac{K}{\gamma-1} (v^{1-\gamma}-\vt^{1-\gamma}) + K \vt^{-\gamma} (v-\vt)  && \text{if} \qu \gamma>1,\\
    &K\big( v/\vt - 1 -\log (v/\vt) \big) && \text{if} \qu \gamma =1.
    \end{aligned}\right.
\end{equation*}
Let $T>0$ be the existence time in \eqref{priori-Rho}. Then by Taylor's theorem, it holds that
\begin{equation}\label{Gvv}
     G[v,\vt] %\sim \phi^2
     \sim  \rho_{+}^{\gamma+1} \phi^2
     \quad \text{ and } \quad \mE[v,u] \sim ( \rho_{+}^{\gamma+1}\phi^2 + \psi^2)  \qquad \text{for } \ (x,t)\in \mQ(T).
\end{equation}

\begin{proposition}[Hardy-type inequality] \label{prop:hardy}
    Suppose $n\ge 2$ and $f(x,t): \mQ(T)\to \R$ is such that $f(-m_bt,t)=0$ for $t\in[0,T]$. Then there exists $C>0$ such that
    \begin{align*}
        \int_{S(t)} \dfrac{f^2}{r^{2n}} \, dx
        \le C \max\{1,\rho_{+}^2\} \int_{S(t)} f_x^2 \, dx.
    \end{align*}
\end{proposition}
\begin{proof}
    From \eqref{priori-Lag} and \eqref{R}, we have
    \begin{align*}
        r(x,t) = \Big( 1+ n \int_{-m_b t}^x v(y,t)\, dy \Big)^{\frac{1}{n}} %\ge  \Big( 1+ \frac{ n}{2\rho_+} \int_{-m_b t}^x  dy \Big)^{\frac{1}{n}} 
        \ge C^{-1} \min\{1,\rho_{+}^{-1}\}^{\frac{1}{n}} (1+ x + m_b t)^{\frac{1}{n}}.
    \end{align*}
    Hence it follows from the change of variable $z:= 1+x+m_b t$ that
    \begin{align}
        \int_{S(t)} \dfrac{f^2}{r^{2n}} \, dx \le& C \max\{1,\rho_{+}^2\}\int_{S(t)} \frac{|f(x,t)|^2}{(1+x+m_b t)^{2}} \, dx\nonumber\\
        =& C\max\{1,\rho_{+}^2\}\int_{1}^{\infty} \frac{f^2}{z^2}\Big\vert_{x=z-m_bt-1} \, dz. \label{temp:minkow}
    \end{align}
    Define the function $F(z,t)\vcentcolon= f(z-m_b t-1,t)$. Then $F(1,t)=0$ for all $t\in[0,T]$ and $F_{z}(z,t) = f_x \vert_{x=z-m_b t-1}$. Applying Minkowski's integral inequality, we get
    \begin{align*}
        &\bigg(\int_{1}^{\infty} \frac{F^2}{z^2} \, d z \bigg)^{\frac{1}{2}}
        %\\ =& \bigg(\int_{1}^{\infty}\!\! \Big( \frac{1}{z}\int_{1}^{z}\!\! F_{z}(\zeta,t) \, d \zeta   \Big)^2 \, d z \bigg)^{\frac{1}{2}}\\ =& \bigg(\int_{1}^{\infty}\!\! \Big( \frac{1}{z}\int_{0}^{z-1}\!\! F_{z}(1+ \xi,t) \, d \xi   \Big)^2 \, d z \bigg)^{\frac{1}{2}}\\ &
        =\bigg(\int_{1}^{\infty}\!\!\bigg(\int_{0}^{1}\! \frac{z-1}{z} F_{z}\big( 1+(z-1)y,t\big)\, dy  \bigg)^2 d z\bigg)^{\frac{1}{2}}\\
        %\le & \int_{0}^1 \bigg( \int_{1}^{\infty} \! \Big|\frac{z-1}{z}\Big|^2 |F_z|^2 \big( 1+(z-1)y , t \big)\, dz \bigg)^{\frac{1}{2}}\, d y\\
         \le & \int_{0}^1 \bigg( \int_{1}^{\infty} \!  F_z^2 \big( 1+(z-1)y , t \big)\, dz \bigg)^{\frac{1}{2}}\, d y
         %\\=& \int_{0}^{1} \bigg( \int_{1}^{\infty} \!\! y^{-1} |F_{z}|^2(\lambda, t) \, d\lambda \bigg)^{\frac{1}{2}} \, d y\\ =& \bigg(\int_{0}^1 y^{-1/2}\, d y \bigg) \bigg(\int_{1}^{\infty}\!\! |F_z(z,t)|^2\, dz \bigg)^{\frac{1}{2}}
         = 2 \bigg(\int_{1}^{\infty}\!\! F_z^2(z,t)\, dz \bigg)^{\frac{1}{2}}.
    \end{align*}
    Finally, using the above inequality in \eqref{temp:minkow} and rewriting in $f(x,t)$, we obtain that
    \begin{align*}
        \int_{S(t)} \dfrac{f^2}{r^{2n}} \, dx 
        \le C \max\{1,\rho_{+}^2\} \int_{1}^{\infty} \!\!F_{z}^2 (z,t)\, dz = C  \max\{1,\rho_{+}^2\} \int_{S(t)}\! f_x^2(x,t)\, dx.
    \end{align*}
    This concludes the proof.
\end{proof}

\begin{lemma}[Relative Energy Estimate] \label{lem:REE}
    Let $(v,u)$ solve \eqref{nsl-hat} in $\mQ(T)$ with $T>0$ such that \eqref{priori-Lag} holds. Then there exists a constant $\ep=\ep(\rho_+,\mu,\gamma,K,n)>0$ such that if $u_b< \ep$ then
	\begin{align*}
            &\dfrac{d}{d t}\!\int_{S(t)}\!\! \mE[v,u]\, d x + \mu\!\!\int_{S(t)}\!\! \Big\{ \dfrac{r^{2(n-1)}\psi_x^2}{v} + \dfrac{(n-1)v\psi^2}{2r^2}   \Big\}\, d x\\
            &\le C \tilde{A}_1(\rho_{+})\Big\{ u_b^2+|\rho_b-\rho_+| + \frac{|\rho_b-\rho_+|^2}{u_b^2} \Big\} \int_{S(t)}\!\!  \phi_x^2 \, d x,
	\end{align*}
	where
	\begin{equation*}
		\tilde{A}_1(\rho_{+}) \vcentcolon =\max\{1,\rho_{+}^2\} \cdot \max\{ \rho_+^{2\gamma+3}, \rho_{+}^{4\gamma-1}, \rho_{+}^{2\gamma}, \rho_{+}^{\gamma+2} \}.
	\end{equation*}
\end{lemma}
\begin{proof}
Taking temporal derivative on $\mE$, it follows from equations \eqref{dLag} that
\begin{align*}
    &\mE[v,u]_t + \mu \dfrac{(r^{n-1}\psi)_x^2}{v} + \Big\{ r^{n-1}\psi \Big( p(v) - p(\vt) - \dfrac{(r^{n-1}\psi)_x}{v} \Big) \Big\}_x\\
    =& -(\gamma-1) m_b \dfrac{\vt_r}{r^{n-1}} G[v,\vt] - \ut_r \psi^2 + \tilde{\mL} \psi \phi,
\end{align*}
where $\tilde{\mL}(r)$ is given in \eqref{mL}. Integrating the above in $x \in [-m_b t,N]$ for some $N\in \mathbb{N}$, it follows from the boundary condition \eqref{ldbc} and 
the 
Leibniz integral rule that
\begin{align}
    &\dfrac{d}{d t}\int_{-m_b t}^{N}\!\!\! \mE[v,u] \, dx + \mu\int_{-m_b t}^{N}\!\!\dfrac{(r^{n-1}\psi)_x^2}{v}\, dx d\tau \nonumber\\
    =& \!\Big\{ r^{n-1}\psi \Big( \dfrac{(r^{n-1}\psi)_x}{v} - p(\vt) +p(v) \Big)  \Big\}\Big\vert_{x=N} \nonumber\\
    &- \int_{-m_b t}^{N}\!\! \Big\{ (\gamma-1)m_b \dfrac{\vt_r}{r^{n-1}} G[v,\vt] + \ut_r \psi^2 -\tilde{\mL} \psi \phi \Big\} \, d x . \label{mE1}
\end{align}
From Lemma \ref{lm:local}, we get that $r^{n-1}\psi(\cdot,t) \in H^1(S(t))$ uniformly in $t\in[0,T]$. Thus $r^{n-1}\psi(x,t)\to 0$ as $x\to \infty$ for each $t\in[0,T]$. Also, $r^{n-1}\psi_{x} \in L^{\infty}(\mQ(T))$. Consequently, we obtain the limit
\begin{align*}
\lim\limits_{N\to\infty} \!\Big\{ r^{n-1}\psi \Big( \dfrac{(r^{n-1}\psi)_x}{v} - p(\vt) +p(v) \Big)  \Big\}\Big\vert_{x=N}= 0.
\end{align*}
Taking limit $N\to\infty$ in \eqref{mE1} and rewriting $(r^{n-1}\psi)_x^2$ using the relation \eqref{Rdiff}, we get
\begin{align}
    &\dfrac{d}{dt}\int_{S(t)}\!\! \mE[v,u] \, dx  + \mu\int_{S(t)}\!\!\Big\{ \dfrac{r^{2(n-1)}\psi_x^2}{v} + (n-1) \dfrac{v \psi^2}{r^2} \Big\}\, dx \nonumber\\
    =& \int_{S(t)}\!\! \Big\{ \tilde{\mL} \psi \phi -(\gamma-1)m_b \dfrac{\vt_r}{r^{n-1}} G[v,\vt] - \ut_r \psi^2 \Big\} \, d x =\vcentcolon \textrm{(I)}+\textrm{(II)}+\textrm{(III)}. \label{mE2}
\end{align}
By Lemma \ref{le:st}, it is verified that
\begin{align*}
    |\tilde{\mL}(r)| %\le& C r^{-2n+1} \Big\{ \rho_{+}^{2\gamma-1}\frac{|\rho_b-\rho_{+}|}{u_b} + \rho_{+}^{\gamma+1} u_b + \rho_{+}\big(1+\rho_{+}^{\gamma+1}\big) u_b^2 + \rho_{+}^{\gamma} u_b |\rho_b-\rho_{+}| \Big\}\\ &+C r^{-2n+1} \big\{ \rho_{+}^2 u_b^3 + \rho_{+} u_b^3 |\rho_b-\rho_{+}| + \rho_{+}^3 u_b^4  \big\}\\
    \le & C r^{-2n+1} \rho_{+}^{\gamma+1} \Big\{ \rho_{+}^{\gamma-2}\frac{|\rho_b-\rho_{+}|}{u_b} +  u_b \Big\}. 
\end{align*}
Using this, \eqref{priori-Lag} and Cauchy-Schwarz's inequality, we get
\begin{align*}
    | \textrm{(I)} | \le& \dfrac{(n-1)\mu}{4} \int_{S(t)} \dfrac{v \psi^2}{r^2} \, dx  + \dfrac{1}{(n-1)\mu} \int_{S(t)} |\tilde{\mL}|^2\dfrac{r^2\phi^2}{v}\, dx \\
    \le& \dfrac{(n-1)\mu}{4} \int_{S(t)} \dfrac{v \psi^2}{r^2} \, dx + C \rho_{+}^{2\gamma+3} \Big\{ u_b^2 + \rho_{+}^{2\gamma-4} \frac{|\rho_b-\rho_+|^2}{u_b^2}\Big\} \int_{S(t)} \frac{\phi^2}{r^{4n-2}}\, dx.
\end{align*}
Using \eqref{Gvv} and the decay estimate \eqref{rho} we get
\begin{align*}
    |\textrm{(II)}| =& \Big|  (\gamma-1)m_b \int_{S(t)} \dfrac{\vt_r}{r^{n-1}} G[v,\vt] \, d x\Big| 
    %\\ \le& C \big\{ \rho_{+}^{2\gamma} |\rho_b-\rho_+| + \rho_{+}^{\gamma+2} u_b^2 + \rho_{+}^{\gamma+1}|\rho_b-\rho_{+}|u_b^2 + \rho_{+}^{\gamma+3} u_b^3 \big\} \int_{S(t)} \frac{\phi^2}{r^{3n-2}}\, dx \\ 
    \le C \rho_{+}^{2\gamma}\big\{ |\rho_b-\rho_+| + \rho_{+}^{2-\gamma} u_b^2 \big\}\int_{S(t)} \frac{\phi^2}{r^{3n-2}}\, dx.
\end{align*}
Using the estimate $|\ut_r(r)|\le C u_b r^{-n}$ from Lemma \ref{le:st} and (\ref{priori-Lag}), we get
\begin{align*}
    |\textrm{(III)}| = \Big| \int_{S(t)} \ut_r \psi^2 \, dx \Big| \le 2 C \rho_+ u_b \iint_{S(t)} \dfrac{v \psi^2}{r^n}\, dx.
\end{align*}
By choosing the upper bound for $u_b$, we obtain that
\begin{align*}
    |\textrm{(III)}| \le \dfrac{(n-1)\mu}{4}\int_{S(t)} \dfrac{v \psi^2}{r^2}\, dx \qquad \text{if } \ u_b \le \dfrac{(n-1)\mu}{8 C \rho_+}.
\end{align*}
Substituting the estimates for $\textrm{(I)}-\textrm{(III)}$ into \eqref{mE2}, we obtain
\begin{align*}
    &\dfrac{d}{d t}\int_{S(\tau)}\!\! \mE[v,u](x,\tau) \, dx + \mu\int_{S(t)}\!\!\Big\{ \dfrac{r^{2(n-1)}\psi_x^2}{v} + \dfrac{n-1}{2} \dfrac{v \psi^2}{r^2} \Big\}\, dx \\ 
    \le&   C \max\{ \rho_+^{2\gamma+3}, \rho_{+}^{4\gamma-1}, \rho_{+}^{2\gamma}, \rho_{+}^{\gamma+2} \} \Big\{ u_b^2 + |\rho_b-\rho_+| +\frac{|\rho_b-\rho_+|^2}{u_b^2} \Big\} \int_{S(t)} \frac{\phi^2}{r^{3n-2}}\, dx\\
    \le&  C A_1(\rho_{+}) \Big\{ u_b^2+ |\rho_b-\rho_{+}| +\frac{|\rho_b-\rho_+|^2}{u_b^2} \Big\} \int_{S(t)}\!\! \phi_x^2\, dx, 
\end{align*}
where in the last line we used the fact that $2n \le 3n-2$ and Proposition \ref{prop:hardy}.
\end{proof}

\subsection{\texorpdfstring{$H^1$}{H1}-estimates of \texorpdfstring{$\phi$}{phi}}\label{ssec:H1phi}
We define the term
\begin{equation*}
    \mF:= \mu \dfrac{\phi_x}{v} - \frac{\psi}{r^{n-1}}.
\end{equation*}
Then from equations \eqref{st} and \eqref{nsl-hat}, it is verified that $\mF$ satisfies
\begin{equation}\label{HEq}
    \mF_t + \frac{\gamma K}{\mu} \dfrac{\mF}{v^{\gamma}} = (n-1) \frac{\psi^2}{r^n} - \gamma \frac{p(v)-p(\vt)}{v-\vt} \frac{\vt_r}{r^{n-1}} \phi + q \frac{\psi}{r^{n-1}},
\end{equation}
where
\begin{equation*}
	q:= - \frac{\gamma K}{\mu} v^{-\gamma} + \frac{(r^{n-1}\ut)_r}{r^{n-1}}  - \frac{\mu}{\rho_b u_b} r^{n-1} \Big( \frac{(r^{n-1}\ut)_r}{r^{n-1}} \Big)_r.
\end{equation*}
Since the equation \eqref{dLag-1} holds at the boundary $\{x=-m_b t\}$, it follows from the condition \eqref{ldbc} that $\phi_t(-m_b t, t) = \psi_x(-m_b t, t) $. Taking derivative on the equation $\phi(-m_b t,t)=0$, it holds that $ \phi_x(-m_b t, t) = m_b^{-1} \phi_t (-m_b t, t) = m_b^{-1} \psi_x (-m_b t,t)$. Therefore $\mF$ satisfies the boundary condition
\begin{equation}\label{Hbdry}
    \mF(-m_b t, t) = \frac{\mu}{v_b} \phi_x(-m_b t,t) = \frac{\mu}{u_b} \psi_x(-m_bt,t).
\end{equation}
In addition, by \eqref{Nt} and \eqref{priori-Lag}, we verify that
\begin{equation}\label{HN}
    \int_{S(t)} r^{2(n-1)} \mF^2(x,t) \, dx \le C(\rho_{+})N^2(t), \qquad \text{for } \ t\in[0,T].
\end{equation}
Rewriting the estimate of Lemma \ref{lem:REE} in terms of $\mF$ and choosing $u_b$ and $|\rho_b-\rho_+|$ small enough, the following corollary is easily verified
\begin{corollary}\label{corol:REE}
     Suppose $(v,u)$ is a solution to \eqref{nsl-hat} in $\mQ(T)$ for $T>0$ such that \eqref{priori-Lag} holds. Then there exists $\ep=\ep(\rho_+,\mu,\gamma,K,n)>0$ such that if $|\rho_b-\rho_+|\le u_b^2$ and $u_b< \ep$ then
     \begin{align*}
     &\dfrac{d}{d t}\!\int_{S(t)}\!\!\!\! \mE[v,u]\, d x + \mu\!\!\int_{S(t)}\!\! \Big\{ \dfrac{r^{2(n-1)}\psi_x^2}{v} + \dfrac{(n-1)v\psi^2}{4 r^2}   \Big\}\, d x\\ \le& C A_1(\rho_{+}) \Big\{u_b^2+|\rho_b- \rho_+| + \frac{|\rho_b-\rho_{+}|^2}{u_b^2}\Big\} \int_{S(t)} \!\! \mF^2\, d x, 
     \end{align*}
     where
     \begin{equation}
     A_1(\rho_{+}) \vcentcolon= \rho_{+}^{-2}\tilde{A}(\rho_{+}) =\max\{1,\rho_{+}^2\} \cdot \max\{ \rho_+^{2\gamma+1}, \rho_{+}^{4\gamma-3}, \rho_{+}^{2\gamma-2}, \rho_{+}^{\gamma} \}. \label{A1}
     \end{equation}
\end{corollary}

\begin{lemma}\label{lemma:Hx}
Set $k:= 2(n-2)$ or $2(n-1)$. Suppose $(v,u)$ is a solution to (\ref{nsl-hat}) in $\mQ(T)$ for $T>0$ such that \eqref{priori-Lag} holds. Then there exists $\ep=\ep(\rho_+,\mu,\gamma,K,n)>0$ such that if $|\rho_b-\rho_+|\le u_b^2$ and $u_b< \ep$ then
\begin{align*}
    &\dfrac{d }{d t} \int_{S(t)}\!\! r^k \mF^2\, d x + \omega(\rho_+)\int_{S(t)}\!\! r^{k}\mF^2 \, d x \\
    \le &  \dfrac{\mu^2 \rho_b}{u_b} \psi_x^2(-m_b t,t) + C(\rho_+)\bigg\{\int_{S(t)} r^{k-2(n-1)} v\psi^2 \, dx + \mD(t)\bigg\}, 
\end{align*}
where
\begin{equation}
	\omega(\rho_{+}) \vcentcolon= \frac{\gamma K }{2^{\gamma}\mu} \rho_+^{\gamma}, \qquad \mD(t):= \int_{S(t)}\!\! \Big\{ \dfrac{r^{2(n-1)}}{v}\psi_x^2 + \dfrac{v \psi^2 }{r^2} \Big\}\, dx .  \label{mD}
\end{equation}
\end{lemma}
\begin{proof}
    Multiplying both sides of \eqref{HEq} with $r^{k}\mF$, then integrating by parts in $x\in S(t)$, we obtain
    \begin{align}
        &\frac{1}{2}\int_{S(t)}\!\! \d_t(r^k \mF^2)\, dx + \dfrac{\gamma K}{\mu} \int_{S(t)}\!\! \dfrac{r^k \mF^2}{v^{\gamma}}\, dx \nonumber\\
        =& \frac{k}{2}\int_{S(t)}\!\! r^{k-1}(\psi+\ut)\mF^2 \,dx + (n-1) \int_{S(t)}\!\! r^{k-n}\psi^2 \mF \, dx\nonumber\\
        & + \int_{S(t)}\!\!  r^{k-n+1} q \psi \mF \, dx - \gamma \int_{S(t)}\!\!  \dfrac{p(v)-p(\vt)}{v-\vt}\vt_r r^{k-n+1}\phi \mF \, dx =\vcentcolon \sum_{i=1}^4 I_i. \label{rmH2}
    \end{align}
    By Leibniz integral rule and the boundary condition \eqref{Hbdry}, we get
    \begin{align}
        \int_{S(t)} \d_t\big(\dfrac{r^k \mF^2}{2}\big)(x,t) \, dx %=& -m_b \dfrac{r^k \mF^2}{2}\Big\vert_{x=-m_b t} +\dfrac{d}{dt} \int_{S(t)}\!\! \dfrac{r^k \mF^2}{2}(x,t) \, dx \nonumber\\
        =& -\dfrac{\mu^2 \rho_b}{2 u_b} \psi_x^2 \Big\vert_{x=-m_b t} + \dfrac{d}{dt} \int_{S(t)}\!\! \dfrac{r^k \mF^2}{2}(x,t) \, dx. \label{leibH2}
    \end{align}
    By the decay estimate \eqref{ur}, inequality \eqref{priori-Rho} and Cauchy-Schwarz's inequality, there exist a constant $\ep=\ep(\rho_+,\mu,\gamma,K,n)>0$ such that if $u_b + N(t) \le \ep$ for $t\in[0,T]$, then
    \begin{align*}
        I_1 %=& \frac{k}{2}\int_{S(t)}\!\! r^{k-1}(\psi+\ut)\mF^2 \,dx\\ 
        %\le& \frac{C u_b}{\rho_{+}^{\gamma}} \int_{S(t)}\frac{r^k \mF^2}{v^{\gamma}}\, dx + \frac{\gamma K}{32 \mu} \int_{S(t)} \frac{r^k \mF^2}{v^{\gamma}}\, dx + \frac{8\mu}{\gamma K \rho_{+}^{2\gamma}} \int_{S(t)} \frac{\psi^2}{r^2} \frac{r^k \mF^2}{v^{\gamma}} \, dx \\
        \le& \dfrac{\gamma K}{16 \mu} \int_{S(t)}\!\! \dfrac{r^k \mF^2}{v^{\gamma}}\,dx  + \frac{C}{\rho_{+}^{2\gamma}} \|\psi^2(\cdot,t)\|_{L^{\infty}(S(t))}  \int_{S(t)}\!\! \dfrac{r^{k} \mF^2}{v^{\gamma}}\, dx \le \dfrac{\gamma K}{8 \mu} \int_{S(t)}\!\! \dfrac{r^k \mF^2}{v^{\gamma}}\,dx,
        \end{align*}
        where we
        have
        used the Sobolev embedding theorem, $\|\psi^2(\cdot,t)\|_{L^{\infty}(S(t))} \le C(\rho_{+}) N(t)$. By the same argument, we also obtain
        \begin{align*}
        I_2 %=& (n-1) \int_{S(t)}\!\! r^{k-n} \psi^2 \mF \, dx\\ \le&  \dfrac{\gamma K}{8 \mu} \int_{S(t)}\!\! \dfrac{r^{k}\mF^2}{v^{\gamma}}\, dx + \dfrac{2\mu}{\gamma K} (n-1)^2 \int_{S(t)}\!\! r^{k-2n} v^{\gamma} \psi^4 \, dx\\ \le&  \dfrac{\gamma K}{8 \mu} \int_{S(t)}\!\! \dfrac{r^{k}\mF^2}{v^{\gamma}}\, dx + C \rho_{+}^{1-\gamma} \|\psi^2(\cdot,t)\|_{L^{\infty}(S(t))} \int_{S(t)}\!\! \dfrac{v\psi^2}{r^2} \, dx\\
        %\le & \dfrac{\gamma K}{8 \mu} \int_{S(t)}\!\! \dfrac{r^{k}\mF^2}{v^{\gamma}}\, dx + C \rho_{+}^{\gamma+1} \int_{S(t)}\!\! \dfrac{v\psi^2}{r^2} \, dx\\
        \le & \dfrac{\gamma K}{8 \mu} \int_{S(t)}\!\! \dfrac{r^{k}\mF^2}{v^{\gamma}}\, dx + C(\rho_{+}) \int_{S(t)}\!\! \dfrac{v\psi^2}{r^2} \, dx.
    \end{align*}
    Next, the conditions \eqref{urg}, \eqref{priori-Rho} and Lemma \ref{le:st} imply that if $|\rho_b-\rho_+|\le u_b^2$ then $\|q\|_{L^{\infty}(\mQ(T))} \le C(\rho_{+})$.
    %\begin{align*} |q|=& \Big| -\frac{\gamma K}{mu} v^{-\gamma} + \frac{(r^{n-1}\ut)_r}{r^{n-1}} - \frac{\mu}{\rho_b u_b} r^{n-1} \Big(\frac{(r^{n-1}\ut)_r}{r^{n-1}}\Big)_r \Big|\\ =& \Big| -\frac{\gamma K}{mu} v^{-\gamma} + \rho_b u_b \frac{\vt_r}{r^{n-1}}  + \frac{\gamma K}{\rho_b u_b} r^{n-1} \rt^{\gamma+1} \vt_r - \rho_b u_b \frac{\vt_r}{r^{n-1}} + (n-1) \rho_b u_b \frac{\vt}{r^n}  \Big|\\ =& \Big| -\frac{\gamma K}{mu} v^{-\gamma}  + \frac{\gamma K}{\rho_b u_b} r^{n-1} \rt^{\gamma+1} \vt_r + (n-1) \rho_b u_b \frac{\vt}{r^n}  \Big| \le C \max\{\rho_{+}^{2\gamma-2},\rho_{+}^{\gamma}\}, \end{align*}
    Using this, we obtain
    \begin{align*}
        I_{3} %=&\int_{S(t)}\!\!  r^{k-n+1} q \psi \mF\, dx \\ \le& \dfrac{\gamma K}{8\mu}\int_{S(t)}\!\!\dfrac{r^k \mF^2}{v^{\gamma}}\,dx + \dfrac{2\mu}{\gamma K}\int_{S(t)}\!\! q^2 r^{k-2n+2} v^{\gamma-1} v \psi^2\, dx\\ \le& \dfrac{\gamma K}{8\mu}\int_{S(t)}\!\!\dfrac{r^k \mF^2}{v^{\gamma}}\,dx + C \max\{ \rho_{+}^{3\gamma-3}, \rho_{+}^{\gamma+1} \} \int_{S(t)}\!\! r^{k-2n+2} v \psi^2\, dx\\
        \le& \dfrac{\gamma K}{8\mu}\int_{S(t)}\!\!\dfrac{r^k \mF^2}{v^{\gamma}}\,dx + C(\rho_{+}) \int_{S(t)}\!\! r^{k-2n+2} v \psi^2\, dx.
    \end{align*}
    Finally, by \eqref{priori-Rho}, the decay estimate \eqref{rho} and the mean value theorem, we have
    \begin{align*}
        I_4 %=& -\gamma \int_{S(t)}\!\! \dfrac{p(v)-p(\vt)}{v-\vt} \vt_r r^{k-n+1} \phi \mF\, dx\\ \le& \dfrac{\gamma K}{16 \mu} \int_{S(t)}\!\! \dfrac{r^{k}\mF^2}{v^{\gamma}}\, dx + \dfrac{4\mu \gamma}{K} \int_{S(t)}\!\! \Big|\dfrac{p(v)-p(\vt)}{v-\vt}\Big|^2 |\vt_r|^2 r^{k-2n+2} v^{\gamma} \phi^2\, dx \\ \le& \dfrac{\gamma K}{16 \mu} \int_{S(t)}\!\! \dfrac{r^{k}\mF^2}{v^{\gamma}}\, dx + C\Big\{ u_b^2 + \rho_{+}^{2\gamma-4}\dfrac{|\rho_b-\rho_+|^2}{u_b^2} \Big\} \int_{S(t)} \rho_{+}^{\gamma+2}  \frac{\phi^2}{r^{4n-2}}\, dx\\
        \le& \dfrac{\gamma K}{16 \mu} \int_{S(t)}\!\! \dfrac{r^{k}\mF^2}{v^{\gamma}}\, dx + C\Big\{ \rho_{+}^{\gamma+2} u_b^2 + \rho_{+}^{3\gamma-2}\dfrac{|\rho_b-\rho_+|^2}{u_b^2} \Big\} \int_{S(t)}   \frac{\phi^2}{r^{4n-2}}\, dx.
    \end{align*}
    Applying Proposition \ref{prop:hardy} and $|\rho_b-\rho_{+}| \le u_b^2$, we obtain that 
    \begin{align*}
        I_4 %\le& \dfrac{\gamma K}{16 \mu} \int_{S(t)}\!\! \dfrac{r^{k}\mF^2}{v^{\gamma}}\, dx + C\max\{1,\rho_{+}^2\}\cdot \Big\{ \rho_{+}^{\gamma+2} + \rho_{+}^{3\gamma-2} \Big\} u_b^2\int_{S(t)} \phi_x^2\, dx\\
        \le& \dfrac{\gamma K}{16 \mu} \int_{S(t)}\!\! \dfrac{r^{k}\mF^2}{v^{\gamma}}\, dx + C(\rho_{+}) u_b^2\int_{S(t)} \phi_x^2\, dx.
    \end{align*}
    Since $\mF\vcentcolon= \mu v^{-1}\phi_x - r^{1-n}\psi$, by \eqref{priori-Rho} and the triangular inequality, we have
    \begin{align*}
        \phi_x^2= \dfrac{v^2}{\mu^2} \Big| \mF + \dfrac{\psi}{r^{n-1}} \Big|^2 \le C \rho_{+}^{-\gamma-2} \dfrac{r^k \mF^2}{v^{\gamma}} + C \rho_{+}^{-1} \dfrac{v\psi^2}{r^2}. 
    \end{align*}
    Thus there exists $\ep=\ep(\rho_{+},\mu,\gamma,K,n)>0$ such that if $|\rho_b-\rho_+|\le u_b^{2}$ and $u_b\le \ep$, then
    \begin{align*}
        I_4 %\le& \dfrac{\gamma K}{16 \mu} \int_{S(t)}\!\! \dfrac{r^{m}\mF^2}{v^{\gamma}}\, dx + C\max\{1,\rho_{+}^2\}\cdot \Big\{ 1 + \rho_{+}^{2\gamma-4} \Big\} u_b^2\int_{S(t)} \frac{r^k \mF^2}{v^{\gamma}}\, dx \\ & + C \max\{1,\rho_{+}^2\}\cdot \big\{ \rho_{+}^{\gamma+1} + \rho_{+}^{3\gamma-3} \big\} u_b^2 \int_{S(t)} \frac{v \psi^2}{r^2}\, dx\\  \le& \dfrac{\gamma K}{8\mu} \int_{S(t)}\!\! \dfrac{r^{m}\mF^2}{v^{\gamma}}\, dx + C \max\{ 1, \rho_{+}^2 \} \cdot \max\{ \rho_{+}^{\gamma+1}, \rho_{+}^{3\gamma-3} \} u_b^2 \int_{S(t)} \dfrac{v\psi^2}{r^2} \, dx\\
        \le& \dfrac{\gamma K}{8\mu} \int_{S(t)}\!\! \dfrac{r^{m}\mF^2}{v^{\gamma}}\, dx + C(\rho_{+}) \int_{S(t)} \dfrac{v\psi^2}{r^2} \, dx.
    \end{align*}
    Substituting \eqref{leibH2} and $I_1$--$I_4$ into \eqref{rmH2}, then using the inequality $v^{-\gamma} \ge 2^{-\gamma} \rho_+^{\gamma}$ from \eqref{priori-Lag}, we obtain the desired estimate.
\end{proof}

\subsection{\texorpdfstring{$H^1$}{H1}-estimates of \texorpdfstring{$\psi$}{psi}}\label{ssec:H1psi}
\begin{proposition}[Weighted Sobolev estimate]\label{prop:psiSobolev}
    Suppose $f(x,t): \mQ(T)\to \R$ is such that $f(-m_bt,t)=0$ for $t\in[0,T]$. Let $k\in \mathbb{N}$. Then for each $\epsilon \in (0,1)$,
    \begin{align*}
        \| r^{\frac{k}{2}} f_x(\cdot,t) \|_{L^{\infty}(S(t))}^2 \le C \max\{ 1, \rho_{+}^{-1} \} \Big(1 + \frac{1}{\epsilon} \Big)  \int_{S(t)}\!\! r^k f_x^2\, dx + \epsilon \int_{S(t)} \dfrac{r^{k+2(n-1)}}{v}f_{xx}^2\, dx.
    \end{align*}
\end{proposition}
\begin{proof}
    Fix $t\in [0,T]$, for $y,\, z\ge -m_b t $, we have by Fundamental Theorem of Calculus, the differential relation \eqref{Rdiff} and \eqref{priori-Lag} that
    \begin{align*}
        r^k f_x^2 (y,t) =& \bigg| r^k f_x^2 (z,t) + \int_{z}^{y}\Big\{ k r^{k-n} v f_x^2 + 2 r^k f_x f_{xx} \Big\}(x,t) \, dx \bigg|\\
        \le & r^k f_x^2 (z,t) + \frac{C}{\rho_{+}} \big\{1+ \epsilon^{-1}\big\} \int_{S(t)}\!\! r^{k-n} f_x^2(x,t)\, dx + \epsilon\int_{S(t)}\!\!\!\! \dfrac{r^{k+2(n-1)}}{v} f_{xx}^2(x,t)\, dx 
    \end{align*}
    Integrating the above in $z\in [-m_bt, 1-m_bt]$ to obtain
    \begin{align*}
        r^k f_x^2 (y,t) \le C \max\{ 1, \rho_{+}^{-1} \} \big(1+ \epsilon^{-1}\big) \int_{S(t)}\!\! r^k f_x^2(x,t)\, dx + \epsilon \int_{S(t)}\!\!\!\! \dfrac{r^{k+2(n-1)}}{v} f_{xx}^2(x,t)\, dx.
    \end{align*}
    Taking supremum over $y\ge -m_b t$ on the above concludes the proof.
\end{proof}

\begin{proposition} \label{prop:psitx}
    Set $k:= 2(n-2)$ or $2(n-1)$. Suppose $(\phi,\psi)$ is a solution to \eqref{dLag} in $\mQ(T)$ for $T>0$ such that \eqref{priori-Lag} holds. Then there exists a positive constant $\ep=\ep(\rho_+,\mu,\gamma,K,n)$ such that if $|\rho_b-\rho_+|\le u_b^2$ and $u_b< \ep$ then
    \begin{gather*}
        k \int_{S(t)} \! \! r^{k-n} v \psi_x \psi_t\, dx \le  C A_2(\rho_{+})\int_{S(t)}\!\! r^k \mF^2\, dx + \dfrac{\mu}{8} \int_{S(t)}\!\!\!\! \dfrac{r^{k+2(n-1)}}{v}\psi_{xx}^2\, dx + C(\rho_{+})\mD(t),
    \end{gather*}
    where $A_2(\rho_{+}) \vcentcolon= \rho_{+}^{2\gamma-1} \max\{1,\rho_{+}^2\}$ and $\mD(t)$ are defined in \eqref{mD}.
\end{proposition}
\begin{proof}
    Rewriting equation \eqref{dLag-2} in terms of $\mF$, we get
    \begin{align}
        \psi_t - \mu \dfrac{r^{2(n-1)}}{v}\psi_{xx} =& 2(n-1)\mu r^{n-2}\psi_x - \dfrac{r^{2(n-1)}}{v} \psi_x \big\{ \mF + \dfrac{\psi}{r^{n-1}} \big\} - \mu \vt_r \frac{r^{n-1}}{v} \psi_x \nonumber\\
        &-(n-1)\mu \dfrac{v \psi}{r^2} - \frac{r^{n-1}v p^{\prime}(v)}{\mu} \big\{ \mF + \frac{\psi}{r^{n-1}} \big\} + \mR_2. \label{psi-t}
    \end{align}
    Multiplying both sides of \eqref{psi-t} by $kr^{k-n} v\psi_x $ and integrating in $x\in S(t)$, we get
    \begin{align*}
        &k\int_{S(t)}\!\! r^{k-n} v \psi_t \psi_x \, d x\nonumber\\
        %=& \underbrace{k\mu \int_{S(t)}\!\! r^{k+n-2} \psi_{xx} \psi_x \, dx}_{\textcolor{red}{\textrm{(i)}}} + \underbrace{2k(n-1)\mu \int_{S(t)}\!\! r^{k-2} v \psi_x^2 \, dx}_{\textcolor{red}{\textrm{(vii)}}}  - \underbrace{k\int_{S(t)}\!\! r^{k+n-2} \psi_x^2 \mF\, dx}_{\textcolor{red}{\textrm{(iv)}}} \nonumber\\ &-\underbrace{k\int_{S(t)}\!\! r^{k-1} \psi_x^2 \psi \, dx}_{\textcolor{red}{\textrm{(iii)}}} - \underbrace{k\mu \int_{S(t)}\!\! \vt_r r^{k-1} \psi_x^2\, dx}_{\textcolor{red}{\textrm{(vii)}}} - \underbrace{k(n-1)\mu \int_{S(t)}\!\! r^{k-n-2} v^2  \psi_x \psi\, dx}_{\textcolor{red}{\textrm{(vi)}}} \nonumber\\ &- \underbrace{\dfrac{k}{\mu}\int_{S(t)} r^{k-1} v^2 p^{\prime}(v) \psi_x \mF\, dx}_{\textcolor{red}{\textrm{(ii)}}} - \underbrace{\dfrac{k}{\mu} \int_{S(t)}\!\! r^{k-n} v^2 p^{\prime}(v)  \psi_x \psi \, dx}_{\textcolor{red}{\textrm{(vi)}}} \nonumber\\ &- \underbrace{k \int_{S(t)}\!\! r^{k-n} \ut_r v \psi_x \psi \, dx}_{\textcolor{red}{\textrm{(vi)}}} + \underbrace{k \int_{S(t)}\!\! r^{k-n} \Big\{ \dfrac{\ut_r \ut}{\vt} - \vt_r \dfrac{p^{\prime}(v)-p^{\prime}(\vt)}{v-\vt} v \Big\}  v \psi_x \phi\, dx}_{\textcolor{red}{\textrm{(v)}}}.\nonumber\\
        =& \underbrace{k\mu \int_{S(t)}\!\! r^{k+n-2} \psi_{xx} \psi_x \, dx}_{\textrm{(i)}} -\underbrace{k\int_{S(t)}\!\! r^{k-1} \psi_x^2 \psi \, dx}_{\textrm{(ii)}} -\underbrace{\dfrac{k}{\mu}\int_{S(t)} r^{k-1} v^2 p^{\prime}(v) \psi_x \mF\, dx}_{\textrm{(iii)}} \\
        & - \underbrace{k\int_{S(t)}\!\! r^{k+n-2} \psi_x^2 \mF\, dx}_{\textrm{(iv)}} + \underbrace{k \int_{S(t)}\!\! r^{k-n} \Big\{ \dfrac{\ut_r \ut}{\vt} - \vt_r \dfrac{p^{\prime}(v)-p^{\prime}(\vt)}{v-\vt} v \Big\}  v \psi_x \phi\, dx}_{\textrm{(v)}} \\
        & -\underbrace{k\int_{S(t)}\!\!\!\Big\{ \dfrac{(n-1)\mu v}{r^2} + \dfrac{v p^{\prime}(v)}{\mu} + \ut_r \Big\} \frac{v \psi_x \psi}{r^{n-k}}\, dx}_{\textrm{(vi)}} + \underbrace{k\mu \int_{S(t)}\!\!\! \big\{ \dfrac{2(n-1)v}{r} - \vt_r  \big\} r^{k-1} \psi_x^2 \, dx}_{\textrm{(vii)}}.
    \end{align*}
    By Cauchy-Schwarz's inequality, we get
    \begin{align*}
        |\textrm{(i)}|+|\textrm{(ii)}| \le& \dfrac{\mu}{16}\! \int_{S(t)}\!\!\! \dfrac{r^{k+2(n-1)}}{v}\psi_{xx}^2\, dx +  C(\rho_{+})\{1+N^2(t)\}\! \int_{S(t)}\!\! \Big\{ \dfrac{r^{2(n-1)}}{v} \psi_x^2 + \frac{v \psi^2}{r^2} \Big\}\, dx, 
    \end{align*}
    where $N(t)$ is defined in \eqref{Nt}. Moreover, we also get
    \begin{align*}
       |\textrm{(iii)}| %=& \bigg| \frac{k}{\mu} \int_{S(t)} r^{k-1} v^2 p^{\prime}(v) \psi_x \mF \, dx \bigg| \\ \le & \int_{S(t)} v^3 |p^{\prime}(v)|^2 r^k \mF^2 dx + \frac{k^2}{4\mu^2}\int_{S(t)} v^2 \frac{r^{k-2}\psi_x^2}{v} \, dx\\
       \le & C\rho_{+}^{2\gamma-1} \int_{S(t)} r^k \mF^2 \, dx + C(\rho_{+}) \int_{S(t)} \frac{r^{2(n-1)}}{v} \psi_x^2 \, dx. 
    \end{align*}
    Next applying H\"older's inequality, \eqref{priori-Lag} and Proposition \ref{prop:psiSobolev}, we have
    \begin{align*}
        |\textrm{(iv)}| \le&  k \| r^{\frac{k}{2}} \psi_x^2(\cdot,t) \|_{L^{\infty}(S(t))}\bigg(\int_{S(t)}\!\! r^{2n-4} \psi_x^2\, dx\bigg)^{\frac{1}{2}} \bigg(\int_{S(t)}\!\! r^{k}\mF^2 \, dx\bigg)^{\frac{1}{2}}\\
        \le & C(\rho_{+}) N(t)  \| r^{\frac{k}{2}} \psi_x^2(\cdot,t) \|_{L^{\infty}(S(t))} \bigg(\int_{S(t)}\!\! \frac{r^{2(n-1)}}{v}\psi_x^2 \, dx\bigg)^{\frac{1}{2}}\\
        % \le & C(\rho_{+}) N(t) \max\{1,\rho_{+}^{-1}\} (1+\epsilon^{-1}) \int_{S(t)} r^k \psi_x^2 \, dx  \\ &+ C(\rho_{+}) N(t) \epsilon \int_{S(t)}\!\! \dfrac{r^{k+2(n-1)}}{v}\psi_{xx}^2 \, dx\\
        \le&  C(\rho_{+}) \int_{S(r)} \frac{r^{2(n-1)}}{v} \psi_x^2 \, dx + \frac{\mu}{16} \int_{S(t)} \!\! \dfrac{r^{k+2(n-1)}}{v} \psi_{xx}^2.
    \end{align*}
    By \eqref{priori-Lag}, Cauchy Schwarz's inequality, the decay estimates in Lemma \ref{le:st} and Proposition \ref{prop:hardy}, it follows that
    \begin{align*}
        |\textrm{(v)}| \le& C  \max\{1,\rho_{+}^2\} \cdot \rho_{+}^{2\gamma} \int_{S(t)} \phi_x^2 \, dx  + C(\rho_{+}) \int_{S(t)}\!\! \dfrac{r^{2(n-1)}}{v} \psi_x^2 \, dx\\
        \le & C \rho_{+}^{2\gamma-1} \int_{S(t)}\!\! r^k \mF^2\, dx + C(\rho_{+}) \int_{S(t)}\!\! \Big\{ \dfrac{r^{2(n-1)}}{v} \psi_x^2 + \dfrac{v\psi^2}{r^2} \Big\} \, dx,
    \end{align*}
    where in the last line we used \eqref{priori-Lag} and $\mF\vcentcolon= \mu v^{-1}\phi_x - r^{1-n}\psi$. Finally, by \eqref{priori-Lag} and Cauchy-Schwarz's inequality, we get
    \begin{align*}
        |\textrm{(vi)}|+|\textrm{(vii)}| \le C(\rho_{+}) \int_{S(t)}\!\! \Big\{ \dfrac{r^{2(n-1)}}{v} \psi_x^2 + \dfrac{v\psi^2}{r^2} \Big\} \, dx.
    \end{align*}
    Combining the estimates for \textrm{(i)}--\textrm{(vii)}, we complete the proof.
\end{proof}

\begin{lemma} \label{lemma:psix}
    Set $k:= 2(n-2)$ or $2(n-1)$. Suppose $(\phi,\psi)$ is a solution to \eqref{dLag} in $\mQ(T)$ for $T>0$ such that \eqref{priori-Lag} holds. Then there exists $\ep=\ep(\rho_+,\mu,\gamma,K,n)>0$ such that if $|\rho_b-\rho_+|\le u_b^2$ and $u_b< \ep$ then
    \begin{align*}
    &\dfrac{d }{d t}\int_{S(t)} r^{k}\psi_x^2 \,d x + m_b \psi_x^2(-m_bt,t) + \dfrac{\mu}{2} \int_{S(t)}\!\!\! \frac{r^{k+2(n-1)}}{v}\psi_{xx}^2 \, d x\\ 
    \le & C A_2(\rho_{+}) \int_{S(t)}\!\!\! r^k \mF^2\,d x + C(\rho_{+}) \int_{S(t)}\!\!\! r^{k-2(n-1)} \psi^2\,d x + C(\rho_{+})\mD(t) , 
    \end{align*}
    where
    \begin{equation}
    A_2(\rho_{+}) \vcentcolon= \rho_{+}^{2\gamma-1} \max\{1,\rho_{+}^2\}.\label{A2}
    \end{equation}
\end{lemma}
\begin{proof}
    Since $\psi\vert_{x=-m_b t}=0$, we get $\psi_t\vert_{x=-m_b t} = m_b \psi_x\vert_{x=-m_b t}$. Thus $$\psi_t\psi_x\vert_{x=-m_b t}=m_b \psi_x^2\vert_{x=-m_b t}.$$ Next, with few lines of computations, it is verified that
    \begin{align*}
        -r^k \psi_t \psi_{xx} %=& -(r^k \psi_t \psi_x)_x + \psi_x (r^k\psi_{xt} + k v r^{k-n}\psi_t )\\
        %=& -(r^k \psi_t \psi_x)_x + \frac{r^k}{2}(\psi_x^2)_t + k r^{k-n} v \psi_x \psi_t\\
        =& -(r^k\psi_t\psi_x)_x + \frac{1}{2}(r^k\psi_x^2)_t - \frac{k}{2} r^{k-1} (\psi+\ut) \psi_x^2 + k r^{k-n} v \psi_x \psi_t.
    \end{align*}
    Integrating in $x\in S(t)$, it follows from Leibniz's integral rule that
    \begin{align}
        -\int_{S(t)}\!\!\! r^k \psi_t \psi_{xx}\, dx 
        %=& \psi_t\psi_x\vert_{x=-m_bt} + \dfrac{1}{2}\int_{S(t)}\!\! (r^k \psi_x^2)_t \, dx \nonumber\\ &- \frac{k}{2}\int_{S(t)}\!\! r^{k-1}(\psi+\ut) \psi_x^2\, dx + k\int_{S(t)}\!\! v r^{k-n} \psi_x \psi_t\, dx  \nonumber\\ =& m_b \psi_x^2(-m_bt,t) + \dfrac{d}{dt}\int_{S(t)}\!\!\dfrac{r^k}{2}\psi_x^2 \, dx - \dfrac{m_b}{2} \psi_x^2(-m_bt,t) \nonumber\\ & - \frac{k}{2}\int_{S(t)}\!\! r^{k-1}(\psi+\ut) \psi_x^2\, dx + k\int_{S(t)}\!\! v r^{k-n} \psi_x \psi_t\, dx \nonumber\\
        =& \dfrac{d}{dt}\int_{S(t)}\!\!\dfrac{r^k}{2}\psi_x^2 \, dx + \dfrac{m_b}{2} \psi_x^2(-m_bt,t)\nonumber\\ & - \frac{k}{2}\int_{S(t)}\!\! r^{k-1}(\psi+\ut) \psi_x^2\, dx + k\int_{S(t)}\!\! v r^{k-n} \psi_x \psi_t\, dx.\label{ptpx}
    \end{align}
    Multiplying both sides of \eqref{psi-t} by $-r^k \psi_{xx}$, integrating in $x\in S(t)$ and then using \eqref{ptpx}, we obtain
    \begin{align}
        &\dfrac{d}{d t} \int_{S(t)}\!\! \dfrac{r^k}{2} \psi_x^2 \, dx + \dfrac{m_b}{2} \psi_x^2(-m_b t, t) + \mu \int_{S(t)}\!\! \dfrac{r^{k+2(n-1)}}{v} \psi_{xx}^2\, dx\nonumber\\
        =& - \underbrace{\dfrac{k}{2}\int_{S(t)}\!\! (\psi+\ut) r^{k-1} \psi_x^2 \, dx}_{\textrm{(i)}} + \underbrace{\dfrac{1}{\mu} \int_{S(t)}\! v p^{\prime}(v) r^{k+n-1} \mF \psi_{xx}\, dx}_{\textrm{(ii)}} + \underbrace{\int_{S(t)}\!\! \dfrac{r^{k+2(n-1)}}{v} \mF \psi_x \psi_{xx}\, dx}_{\textrm{(iii)}} \nonumber\\
        & + \underbrace{\int_{S(t)}\!\! \dfrac{r^{k+n-1}}{v} \psi \psi_x \psi_{xx} \, dx}_{\textrm{(iv)}} - \underbrace{\int_{S(t)} \mR_2 r^k \psi_{xx}\, dx}_{\textrm{(v)}} + \underbrace{\mu \int_{S(t)}\!\! \big\{ \dfrac{\vt_r}{v} - \frac{2(n-1)}{r} \big\} r^{k+n-1} \psi_x \psi_{xx}\, dx}_{\textrm{(vi)}} \nonumber\\ 
        & + \underbrace{\int_{S(t)}\!\! \big\{ \frac{v p^{\prime}(v)}{\mu} + \dfrac{(n-1)\mu v}{r^2} \big\} r^{k} \psi_{xx} \psi\, dx}_{\textrm{(vii)}} + \underbrace{k \int_{S(t)}\! v r^{k-n} \psi_x \psi_t \, dx}_{\textrm{(viii)}}.\label{ddtpsix} 
    \end{align}
    Using the decay estimate \eqref{rho}--\eqref{ur} and Cauchy-Schwarz's inequality, we get
    \begin{align*}
        |\textrm{(i)}| + |\textrm{(iv)}| +|\textrm{(vi)}| \le & C(\rho_{+})\big(1+N^2(t)\big) \mD(t) + \dfrac{\mu}{16} \int_{S(t)}\!\!\!\dfrac{r^{k+2(n-1)}}{v} \psi_{xx}^2 \, dx .
    \end{align*}
    Similarly by the same argument, we also get
    \begin{align*}
        |\textrm{(vii)}| \le C(\rho_{+}) \mD(t) + C(\rho_{+})\int_{S(t)} r^{k-2(n-1)} \psi^2 \, dx  + \dfrac{\mu}{16} \int_{S(t)}\!\!\!\dfrac{r^{k+2(n-1)}}{v} \psi_{xx}^2 \, dx .
    \end{align*}
    Using \eqref{priori-Rho} and Cauchy-Schwarz's inequality, we have
    \begin{align*}
        |\textrm{(ii)}| %=& \bigg| \frac{1}{\mu} \int_{S(t)} v p^{\prime}(v) r^{k+n-1} \mF \psi_{xx} \, dx \bigg|\\ \le & \frac{\mu}{16}\int_{S(t)} \frac{r^{k+2(n-1)}}{v} \psi_{xx}^2 \, dx + \frac{4}{\mu^2} \int_{S(t)} v^3 |p^{\prime}(v)|^2 r^k \mF^2 \, dx \\
        \le & \frac{\mu}{16}\int_{S(t)} \frac{r^{k+2(n-1)}}{v} \psi_{xx}^2 \, dx + C \rho_{+}^{2\gamma-1} \int_{S(t)} r^k \mF^2 \, dx.
    \end{align*}
    By H\"older's inequality, \eqref{HN} and Proposition \ref{prop:psiSobolev}, there exists a positive constant $\ep=\ep(\rho_+,\mu,\gamma,K,n)$ such that, if $N(t)\le \ep$ then
    \begin{align*}
        |\textrm{(iii)}| \le& C(\rho_{+})\|r^{\frac{k}{2}} \psi_x\|_{L^{\infty}(S(t))}^2 \int_{S(t)}\!\! r^{2(n-1)} \mF^2\, dx   + \dfrac{\mu}{16} \int_{S(t)}\!\! \dfrac{r^{k+2(n-1)}}{v} \psi_{xx}^2\, dx\\
        %\le& C(\rho_{+}) N^2(t) \bigg\{ \dfrac{1 }{\epsilon} \int_{S(t)}\!\! r^{k} \psi_x^2 \, dx + \epsilon \int_{S(t)}\!\!\!\dfrac{r^{k+2(n-1)}}{v} \psi_{xx}^2 \, dx \bigg\} + \dfrac{\mu}{16} \int_{S(t)}\!\!\! \dfrac{r^{k+2(n-1)}}{v} \psi_{xx}^2\, dx\\
        \le&  C(\rho_{+}) \mD(t) + \frac{\mu}{8} \int_{S(t)}\!\!\!\dfrac{r^{k+2(n-1)}}{v} \psi_{xx}^2 \, dx.
    \end{align*}
    Using the decay estimate \eqref{rho}--\eqref{ur} and Proposition \ref{prop:hardy}, we find a constant $\ep=\ep(\rho_{+},\mu,\gamma,K,n)>0$ such that if $u_b+ |\rho_b-\rho_{+}| \le \ep$ then
    \begin{align*}
        |\textrm{(v)}| \le& \frac{C}{\rho_{+}} \int_{S(t)}\!\! r^{k-2(n-1)}\mR_2^2 \, dx + \dfrac{\mu}{8} \int_{S(t)}\!\! \dfrac{r^{k+2(n-1)}}{v} \psi_{xx}^2\, dx \\
        \le& C \rho_{+}^{2\gamma} \int_{S(t)} \dfrac{\phi^2}{r^{2n}}\, dx + C(\rho_{+}) \int_{S(t)} \dfrac{v\psi^2}{r^2} \, d x +  \dfrac{\mu}{8} \int_{S(t)}\!\! \dfrac{r^{k+2(n-1)}}{v} \psi_{xx}^2\, dx\\
        \le & C \rho_{+}^{2\gamma-1} \max\{1,\rho_{+}^2\} \int_{S(t)} r^k \mF^2 \, dx + C(\rho_{+}) \mD(t) +  \dfrac{\mu}{8} \int_{S(t)}\!\! \dfrac{r^{k+2(n-1)}}{v} \psi_{xx}^2\, dx.
    \end{align*}
    The estimate for the term \textrm{(viii)} is given by Proposition \ref{prop:psitx} as
    \begin{align*}
        |\textrm{(viii)}| \le C A_2(\rho_{+})\int_{S(t)}\!\! r^k \mF^2\, dx + \dfrac{\mu}{8} \int_{S(t)}\!\!\!\! \dfrac{r^{k+2(n-1)}}{v}\psi_{xx}^2\, dx + C(\rho_{+})\mD(t).
    \end{align*}
    Substituting \textrm{(i)}--\textrm{(viii)} in \eqref{ddtpsix}, we obtain the desired inequality.
\end{proof}

\begin{lemma}\label{lemma:2n-4}
    Suppose $(\phi,\psi)$ is a solution to \eqref{dLag} in $\mQ(T)$ for $T>0$ such that \eqref{priori-Lag} holds. Then there exist positive constants $\alpha=\alpha(\mu,\gamma,K,n)$ which is independent of $\rho_+$ and $\ep=\ep(\rho_+,\mu,\gamma,K,n)$, such that if $\rho_+\in (0,\alpha]$, $|\rho_b-\rho_+|\le u_b^2$ and $u_b< \ep$, then
    \begin{align*}
         m_b\int_{0}^{T}\!\!\!\psi_x^2(-m_bt,t)\,\dif t + \int_{0}^{T}\!\!\int_{S(t)}\!\! \Big\{ r^{2(n-2)}\phi_x^2 +  r^{4n-6} \psi_{xx}^2 \Big\} \, d x d t \le  C(\rho_{+}) N^2(0).
    \end{align*}
\end{lemma}
\begin{proof}
Taking $k=2(n-2)$ in Lemma \ref{lemma:Hx} and \ref{lemma:psix}, we obtain that
\begin{gather*}
    \dfrac{d }{d t} \int_{S(t)}\!\! r^{2(n-2)} \mF^2\, d x + \omega(\rho_{+})\int_{S(t)}\!\! r^{2(n-2)}\mF^2 \, d x \le  \dfrac{\mu^2 \rho_b}{u_b} \psi_x^2(-m_b t,t)  + C(\rho_{+}) \mD(t), \\
    \dfrac{d }{d t}\int_{S(t)}\!\!\! r^{2(n-2)}\psi_x^2\,d x  + \dfrac{\mu}{2} \int_{S(t)}\!\!\!\! \frac{r^{4n-6}}{v}\psi_{xx}^2\, d x \le   C A_2(\rho_{+}) \int_{S(t)} \!\!\! r^{2(n-2)} \mF^2\,d x + C(\rho_+)\mD(t) . 
\end{gather*}
Multiplying the second inequality by $\big(2CA_2\big)^{-1} \omega$ and adding it to the first, we get
\begin{align}
    &\dfrac{d }{d t} \int_{S(t)}\!\! r^{2(n-2)} \Big\{ \mF^2 + \frac{\omega}{2C A_2} \psi_x^2 \Big\}\, d x  + \int_{S(t)}\!\! \Big\{ \dfrac{\omega}{2}r^{2(n-2)}\mF^2 + \dfrac{\mu\omega}{4C A_2} \dfrac{}{} \dfrac{r^{4n-6}}{v} \psi_{xx}^2 \Big\} \, d x \nonumber\\
    & \le \dfrac{\mu^2 \rho_b}{u_b} \psi_x^2(-m_b t,t)  + C(\rho_{+}) \mD(t). \label{temp:Fpsix}
\end{align}
Applying \eqref{priori-Lag} and Proposition \ref{prop:psiSobolev} with $\epsilon=(8C \mu \rho_b A_2)^{-1}u_b\omega$, we have
\begin{align*}
    \psi_x^2(-m_b t,t) \le& C \max\{1,\rho_{+}^{-1}\} \Big( 1 + \frac{\rho_{+} A_2}{u_b \omega} \Big) \int_{S(t)}\!\!\!\! \dfrac{r^{2(n-1)}\psi_x^2}{v}  dx +  \dfrac{u_b \omega}{8C \mu \rho_b A_2}\int_{S(t)}\!\!\!\! \dfrac{r^{4n-6}\psi_{xx}^2}{v} dx.
\end{align*}
Recalling the definition of $A_2$ in \eqref{A2} and $\omega$ in \eqref{mD}, we obtain that
\begin{equation}
    \dfrac{\mu^2 \rho_b}{u_b} \psi_x^2(-m_bt,t) \le \dfrac{\mu\omega}{8C A_2}\int_{S(t)}\!\!\!\! \dfrac{r^{4n-6}}{v}\psi_{xx}^2\, dx + C(\rho_{+}) \frac{\mD(t)}{u_b} + C \frac{A_3(\rho_{+})}{u_b^2}  \mD(t) , \label{bdryPsix2}
\end{equation}
where
\begin{equation}
A_3(\rho_+) \vcentcolon= \max\{1,\rho_{+}\}  \frac{\rho_{+}A_2(\rho_{+})}{\omega(\rho_{+})} = \frac{2^\gamma \mu}{\gamma K} \rho_{+}^{\gamma} \max\{1,\rho_{+}^3\}.\label{A3} 
\end{equation}
Substituting this in \eqref{temp:Fpsix}, we get
\begin{align*}
    &\dfrac{d }{d t} \int_{S(t)}\!\! r^{2(n-2)} \Big\{ \mF^2 + \frac{\omega}{2C A_2} \psi_x^2 \Big\}\, d x  + \int_{S(t)}\!\! \Big\{ \dfrac{\omega}{2}r^{2(n-2)}\mF^2 + \dfrac{\mu\omega}{8C A_2} \dfrac{}{} \dfrac{r^{4n-6}}{v} \psi_{xx}^2 \Big\} \, d x \nonumber\\
    & \le C \frac{A_3(\rho_{+})}{u_b^2} \mD(t)  + C(\rho_{+}) \Big\{ 1+ \frac{1}{u_b} \Big\} \mD(t).
\end{align*}
Integrating the above in time and using Corollary \ref{corol:REE}, we obtain that
\begin{align*}
    &\int_{S(\tau)}\!\! r^{2(n-2)} \Big\{ \mF^2 + \frac{\omega}{2C A_2} \psi_x^2 \Big\}\, d x \bigg\vert_{\tau=0}^{\tau=t} + \int_{0}^{t}\!\!\int_{S(\tau)}\!\! \Big\{ \dfrac{\omega}{2}r^{2(n-2)}\mF^2 + \dfrac{\mu \omega}{8C A_2} \dfrac{r^{4n-6}}{v} \psi_{xx}^2 \Big\} \, d x d\tau \\
    \le& C \frac{A_3(\rho_{+})}{u_b^2} \cdot A_1(\rho_{+}) \Big\{ u_b^2 + |\rho_b-\rho_+| + \frac{|\rho_b-\rho_+|^2}{u_b^2} \Big\} \int_{0}^{t}\!\!\int_{S(\tau)} r^{2(n-2)} \mF^2 d x d \tau  \\
    &+ C(\rho_{+}) \Big\{1+ \dfrac{1}{u_b} \Big\} \Big\{ u_b^2 + |\rho_b-\rho_+| + \frac{|\rho_b-\rho_+|^2}{u_b^2} \Big\} \int_{0}^{t}\!\!\int_{S(\tau)} r^{2(n-2)} \mF^2 d x d \tau.
\end{align*}
First, set $|\rho_b-\rho_+|\le u_b^2$. Then there exists a constant $\ep=\ep(\rho_+,\mu,\gamma,K,n)>0$ such that if $u_b + |\rho_b-\rho_+|\le \ep$, then
\begin{align}
   &\int_{S(\tau)}\!\! r^{2(n-2)} \Big\{ \mF^2 + \frac{\omega}{2C A_2} \psi_x^2 \Big\}\, d x \bigg\vert_{\tau=0}^{\tau=t} + \int_{0}^{t}\!\!\int_{S(\tau)}\!\! \Big\{ \dfrac{\omega}{2}r^{2(n-2)}\mF^2 + \dfrac{\mu \omega}{8C A_2} \dfrac{r^{4n-6}}{v} \psi_{xx}^2 \Big\} \, d x d\tau \nonumber\\
    \le& C A_3(\rho_{+}) A_1(\rho_{+})\int_{0}^{t}\!\!\int_{S(\tau)} r^{2(n-2)} \mF^2 d x d \tau + \frac{\omega(\rho_+)}{8}\int_{0}^{t}\!\!\int_{S(\tau)} r^{2(n-2)} \mF^2 d x d \tau. \label{2n-4:temp1}
\end{align}
We emphasize that the constant $C=C(\mu,\gamma,K,n)>0$ appearing in the above inequality is independent of $\rho_{+}$. By the definition of $A_1$ in \eqref{A1}  and $A_3$ in \eqref{A3},
\begin{align*}
    A_3(\rho_{+}) A_1(\rho_{+}) %=& \frac{2^\gamma \mu}{\gamma K} \rho_{+}^{\gamma} \max\{1,\rho_{+}^3\} \cdot \max\{1,\rho_{+}^2\} \cdot \max\{ \rho_+^{2\gamma+1}, \rho_{+}^{4\gamma-3}, \rho_{+}^{2\gamma-2}, \rho_{+}^{\gamma} \}\\
    =& \frac{2^\gamma \mu}{\gamma K} \max\{1,\rho_{+}^5\} \cdot \max\{ \rho_+^{3\gamma+1}, \rho_{+}^{5\gamma-3}, \rho_{+}^{3\gamma-2}, \rho_{+}^{2\gamma} \}.
\end{align*}
Since $\gamma>1$, the above expression implies that there is a constant $\alpha=\alpha(\mu,\gamma,K,n)>0$ which is independent of $\rho_+$ such that if $\rho_{+}\in (0,\alpha]$ then
\begin{equation*}
    C A_3(\rho_{+}) A_1(\rho_{+}) \le \frac{\gamma K }{2^{\gamma+3}\mu} \rho_+^{\gamma} =\frac{\omega(\rho_{+})}{8}.
\end{equation*}
Substituting the above into \eqref{2n-4:temp1}, we obtain
\begin{align}
   &\int_{S(\tau)}\!\! r^{2(n-2)} \Big\{ \mF^2 + \frac{\omega}{2C A_2} \psi_x^2 \Big\}\, d x \bigg\vert_{\tau=t} + \int_{0}^{t}\!\!\int_{S(\tau)}\!\! \Big\{ \dfrac{\omega}{4}r^{2(n-2)}\mF^2 + \dfrac{\mu \omega}{8C A_2} \dfrac{r^{4n-6}}{v} \psi_{xx}^2 \Big\} \, d x d\tau \nonumber\\
    &\le  \int_{S(\tau)}\!\! r^{2(n-2)} \Big\{ \mF^2 + \frac{\omega}{2C A_2} \psi_x^2 \Big\}\, d x \bigg\vert_{\tau=0} \le C(\rho_+) N^2(0). \label{2n-4:temp2}
\end{align}
Using $\mF\vcentcolon= \mu v^{-1} \phi_x - r^{1-n} \psi$, we apply triangular inequality in \eqref{2n-4:temp2} to get
\begin{align*}
      \sup\limits_{0\le t\le T}\int_{S(t)}\!\! r^{2(n-2)} \big\{ \phi_x^2 + \psi_x^2 \big\}\, d x +\iint_{\mQ(T)}\!\! \Big( r^{2(n-2)}\phi_x^2 +  r^{4n-6} \psi_{xx}^2 \Big) \, d x d t & \le C(\rho_{+}) N^2(0).
\end{align*}
Moreover, combining the above estimate with Corollary \ref{corol:REE}, we obtain that
\begin{align}
    &\sup\limits_{0\le t\le T} \int_{S(t)}\!\! \mE[v,u]\, d x + \mu\!\!\iint_{\mQ(T)}\!\! \Big\{ r^{2(n-1)}\psi_x^2 + \dfrac{\psi^2}{ r^2}   \Big\}\, d x \le C(\rho_{+}) N^2(0). \label{REE-final}
\end{align}
Finally the estimate for $m_b\int_{0}^T \psi_x^2 (-m_bt,t)dt$ is obtained by integrating \eqref{bdryPsix2} in $t\in[0,T]$ and substituting \eqref{2n-4:temp2}--\eqref{REE-final} in the resultant inequality.
\end{proof}

\begin{lemma}
    Suppose $(\phi,\psi)$ is a solution to \eqref{dLag} in $\mQ(T)$ for $T>0$ such that \eqref{priori-Lag} holds. Then there exist constants $\alpha=\alpha(\mu,\gamma,K,n)>0$ which is independent of $\rho_+$ and $\ep=\ep(\rho_+,\mu,\gamma,K,n)>0$, such that if $\rho_+\in (0,\alpha]$, $|\rho_b-\rho_+|\le u_b^2$ and $u_b< \ep$, then
    \begin{align*}
         & \sup\limits_{0 \le t\le T}\int_{S(t)}\!\!  r^{2(n-1)} \big\{  \phi_x^2 + \psi_{x}^2 \big\}(x,t) \, d x \le C(\rho_{+}) \Big\{1+ \frac{1}{u_b^2}\Big\}  N^2(0).
    \end{align*}
\end{lemma}
\begin{proof}
Taking $k=2(n-1)$ in Lemmas \ref{lemma:Hx}, it follows from \eqref{priori-Lag} that
\begin{gather*}
    \begin{aligned}
        &\dfrac{d }{d t} \int_{S(t)}\!\! r^{2(n-1)} \mF^2\, d x + \omega(\rho_{+})\int_{S(t)}\!\! r^{2(n-1)}\mF^2 \, d x\\ 
        \le& C(\rho_+)\int_{S(t)}\!\!\psi^2\, dx + C(\rho_+)\tilde{\mD}(t) \le C N^2(0) + C\tilde{\mD}(t), 
    \end{aligned} 
\end{gather*}
where $\tilde{\mD}(t) \vcentcolon= \mD(t) + \frac{\rho_b}{u_b} \psi_x^2(-m_b t,t)$. It follows from Lemma \ref{lemma:2n-4} that 
\begin{equation*}
    u_b^2 \int_{0}^T \! \tilde{\mD}(t)\, d t \le C(\rho_{+}) N^2(0).
\end{equation*}
Thus applying Gr\"onwall's lemma, we obtain
\begin{align*}
    \sup\limits_{0\le t\le T}\int_{S(t)}\!\! r^{2(n-1)} \mF^2\, dx \le C(\rho_{+}) \Big\{ 1 + \dfrac{1}{u_b^2}  \Big\} N^2(0). 
\end{align*}
By \eqref{REE-final} and triangular inequality, we get
\begin{align*}
    \sup\limits_{0\le t\le T}\int_{S(t)}\!\! r^{2(n-1)} \phi_x^2\, dx \le C(\rho_+) \Big\{ 1 + \dfrac{1}{u_b^2}  \Big\} N^2(0).
\end{align*}
In addition, combining Lemma \ref{lemma:psix} with $k=2(n-1)$ and Corollary \ref{corol:REE} gives
\begin{align*}
        %&\dfrac{d }{d t}\int_{S(t)} r^{2(n-1)}\psi_x^2(x,t)\,d x + m_b \psi_x^2(-m_bt,t) + C^{-1}  \int_{S(t)} r^{4(n-1)}\psi_{xx}^2(x,t)\, d x\\  \le & C\int_{S(t)}\!\! \psi^2(x,t)\,d x + C\int_{S(t)} r^{2(n-1)} \mF^2\,d x + C\mD(t) , \\
        %&\dfrac{\dif}{\dif t}\!\int_{S(t)}\!\!\!\! \mE[v,u]\, d x + \mu\!\!\int_{S(t)}\!\! \Big\{ \dfrac{r^{2(n-1)}\psi_x^2}{v} + \dfrac{(n-1)v\psi^2}{4 r^2}   \Big\}\, d x \le C\big\{u_b^2+|\rho_b-\rho_+|\big\} \int_{S(t)} \!\! \mF^2 \, d x,\\
        &\dfrac{d}{dt} \int_{S(t)}\!\!\! \big(  \mE[v,u] + r^{2(n-1)} \psi_x^2\big)\, dx + C(\rho_+)^{-1}\int_{S(t)}\!\!\! \big(  \mE[v,u] + r^{2(n-1)} \psi_x^2\big)\, dx\\
        &\le C(\rho_+) \bigg\{ \int_{S(t)}\!\! \mE[v,u]\,d x + \int_{S(t)} r^{2(n-1)} \mF^2\,d x + \mD(t) \bigg\} \le C(\rho_+) \big\{ N^2(0) + \mD(t) \big\},
\end{align*}
Therefore by the Gr\"onwall lemma used previously, we also obtain
\begin{align*}
    \sup\limits_{0\le t\le T}\int_{S(t)}\!\! r^{2(n-1)} \psi_x^2\, dx \le C(\rho_+) N^2(0).
\end{align*}
This concludes the proof and the proof of Theorem \ref{thm:Lag}.
\end{proof}

To extend the solution beyond the maximal time of existence with the help of Theorem \ref{thm:Lag}, we use the local-in-time well-posedness, Lemma \ref{lm:local}. For this, we need to recover the H\"older regularity of $(\rho,u)$, which is stated in the following.
\begin{proposition}\label{prop:schauder}
    Denote $Q_T\vcentcolon=[1,\infty)\times[0,T]$. Let initial data $(\rho_0,u_0)$ satisfy the same assumption of Theorem \ref{mth}. Suppose $(\rho,u)$ is a solution to the system \eqref{ns}, \eqref{stbdry} and \eqref{u+} in the time interval $t\in[0,T]$ for some $T>0$. Then there exists a constant $C_0>0$ independent of $T>0$ such that
    \begin{equation*}
        \| \rho \|_{\mB^{1+\sigma/2,1+\sigma}(Q_T)} + \| u \|_{\mB^{1+\sigma/2,2+\sigma}(Q_T)} \le C_0 e^{C_0 T}.
    \end{equation*}
\end{proposition}
The above proposition is proved using Schauder theory for the parabolic equations (See \cite{friedman64} and \cite{lady}). In \cite{k-n-z03}, the same argument is used to prove the H\"older regularity of solutions to the $1$ dimensional isentropic outflow problem. Since Proposition \ref{prop:schauder} is shown following a procedure similar to the one given in \cite{k-n-z03}, we skip its proof in the present paper. For more examples on the applications of Schauder theory to different problems related to the compressible Navier Stokes equations, we refer readers to the papers \cite{k-n81}, \cite{NNY} and \cite{NN}.

Next, we show the global-in-time existence of $(\rho,u)$. Recall that $N_E$, $M_E$ are the energy norms in Eulerian coordinate variables, defined in \eqref{NE}--\eqref{ME}. Then by \eqref{NEL} and Theorem \ref{thm:Lag}, it follows that 
\begin{equation}\label{mth-Euler}
    \sup\limits_{0\le t \le T} N_{E}^2(t) + M_{E}^2(T) \le C \Big\{ 1 + \frac{1}{u_b^2} \Big\} N_E^2(0). 
\end{equation}
Fix $\rho_+\in (0,\alpha]$ and $0<u_b + |\rho_b-\rho_+| \le \ep $. Applying the standard continuity argument on \eqref{mth-Euler}, we find $\ep_0=\ep_0(u_b,\rho_+,\mu,\gamma,K,n)>0$ such that if $N_E(0)\le \ep_0$ then the solution $(\rho,u)$ exists globally in time with $\sup_{ t \ge 0 }N_{E}(t) + M_E(\infty) <\infty$. 

To conclude, the asymptotic convergence \eqref{Tasymp} is obtained as follows. 
Applying the Cauchy-Schwarz's inequality and the estimate \eqref{mth-Euler}, we get for all $t\ge 0$,
\begin{align*}
&\sup_{r\ge 1} |u(r,t)-\ut(r)|^2
%\\ \le& C \Big( \int_1^\infty\!\! \frac{(u-\ut)^2}{r^{n-1}} \, \dif r \Big)^{\frac{1}{2}} \Big( \int_1^\infty\!\! r^{n-1} (u-\ut)_r^2 \, \dif r \Big)^{\frac{1}{2}}
\le C \Big( \int_1^\infty \frac{(u-\ut)^2}{r^{n-1}}(r,t) \, \dif r \Big)^{\frac{1}{2}},\\
&\sup_{r \ge 1} |\rho(r,t)-\rt(r)|^2 %\\ \le & C \Big( \int_1^\infty (\rho-\rt)^2r^{n-1} \, \dif r \Big)^{\frac{1}{2}} \Big( \int_1^\infty \dfrac{(\rho-\rt)_r^2}{r^{n-1}} \, \dif r \Big)^{\frac{1}{2}}
\le C \Big( \int_1^\infty \frac{(\rho-\rt)_r^2}{r^{n-1}}(r,t) \, \dif r \Big)^{\frac{1}{2}}.
\end{align*}
Thus, in order to prove \eqref{Tasymp}, it suffices to show that as $t\to \infty$,
\begin{equation}
\mathcal{I}_1(t)
\vcentcolon=
\int_1^\infty \frac{(u-\ut)^2}{r^{n-1}}(r,t) \, \dif r
\to 0, 
\qquad 
\mathcal{I}_2(t)
\vcentcolon=
\int_1^\infty \frac{(\rho-\rt)_r^2}{r^{n-1}}(r,t) \, \dif r
\to 0.
\label{temp:I1I2}
\end{equation}
To verify \eqref{temp:I1I2}, we calculate $d\mathcal{I}_{i}/dt$, for $i=1,\,2$, using the equations \eqref{ns}. Then we show that $\mathcal{I}_i$, $|d\mathcal{I}_{i}/dt| \in L^1(0,\infty)$ for $i=1,\,2$. The detail for this is abbreviated since the procedure is standard   and the same derivations are found in \cite{k-n-z03}, \cite{k-n81}, \cite{NNY}
and
\cite{NN}. 

\section*{Acknowledgement}
The authors would like to thank the referee for their attentive revision and thoughtful suggestions on the paper.

\end{document}